\PassOptionsToPackage{hypertexnames=false}{hyperref}
\documentclass[pdflatex,sn-mathphys-num]{sn-jnl}

\usepackage{graphicx}%
\usepackage{multirow}%
\usepackage{amsmath,amssymb,amsfonts}%
\usepackage{amsthm}%
\usepackage{mathrsfs}%
\usepackage[title]{appendix}%
\usepackage{xcolor}%
\usepackage{textcomp}%
\usepackage{manyfoot}%
\usepackage{booktabs}%
\usepackage{algorithm}%
\usepackage{algorithmicx}%
\usepackage{algpseudocode}%
\usepackage{listings}%

\theoremstyle{thmstyleone}%
\newtheorem{theorem}{Theorem}
\newtheorem{proposition}[theorem]{Proposition}%

\theoremstyle{thmstyletwo}%
\newtheorem{example}{Example}%
\newtheorem{remark}{Remark}%

\theoremstyle{thmstylethree}%
\newtheorem{definition}{Definition}%

\usepackage{microtype}
\usepackage{booktabs} 

\usepackage{hyperref}
\usepackage{xurl}
\usepackage{amsopn}
\usepackage{breakcites}

\usepackage{amssymb}
\usepackage{xcolor}

\DeclareMathOperator{\sign}{sign}
\usepackage{mathtools}
\usepackage{booktabs}
\usepackage{multirow}

\newtheorem{lemma}[theorem]{Lemma}
\newtheorem{corollary}[theorem]{Corollary}

\let\REQUIRE\Require
\let\FOR\For
\let\STATE\State
\let\IF\If
\let\ENDIF\EndIf
\let\ENDFOR\EndFor
\let\RETURN\Return

\newcommand{\N}{\mathbb{N}}
\newcommand{\R}{\mathbb{R}}

\newcommand{\prox}[3][]{\operatorname{prox}^{#1}_{#2}\left(#3 \right)}

\DeclareMathAlphabet{\mathscr}{OMS}{cmsy}{m}{n}

\begin{document}

\title[Variable Smoothing for Weakly Convex Problems with Non-Euclidean Directions]{Variable Smoothing for Weakly Convex Problems with Non-Euclidean Directions}


\author*[1]{\fnm{Farid} \sur{Najar}}\email{farid.najar@outlook.com}


\affil[1]{\orgdiv{DAVID Lab}, \orgname{UVSQ}, \orgaddress{\street{45 Avenue des États-Unis}, \city{Versailles}, \postcode{78000}, \state{Yvelines}, \country{France}}}


\abstract{
We propose MELMO (Moreau Envelope Smoothing with Linear Minimization Oracles), an algorithm for composite optimization problems of the form $\min_x f(x) + g(Tx)$, where $f$ is smooth and $g$ may be non-smooth. The method leverages the Moreau envelope to smooth the non-smooth component while adapting to problem geometry through linear minimization oracles. Assuming $g$ is $\rho$-weakly convex, we establish a family of convergence bounds parameterized by the step-size and smoothing schedules, thereby making explicit the trade-off between optimizing the smoothed objective and recovering stationarity for the original composite problem. In particular, one regime yields $\mathcal O(k^{-1/4})$ rates for both the smoothed-gradient norm and a composite stationarity proxy, while another yields $\mathcal O(k^{-1/3})$ for the smoothed-gradient norm together with $\mathcal O(k^{-1/4})$ for the composite proxy.
We also establish a $K$-horizon-dependent convergence rate that yields $\mathcal O(K^{-1/3})$ for the composite proxy.
Empirically, MELMO is competitive with variable smoothing and subgradient baselines on sparse low-rank matrix factorization and image denoising.
}




\maketitle

\section{Introduction}
Optimization problems involving non-smooth objective functions arise frequently in machine learning, signal processing, and statistics. These problems often take the form of composite minimization, where the objective function is the sum of a smooth data fidelity function (e.g., MSE) and a non-smooth regularizer that encodes some inductive bias for the sought-after solution (e.g., sparsity, low rank, etc.). We consider the following composite minimization problem

\begin{equation}\tag{P}
\min_{x \in \mathbb{R}^n} F(x) = f(x) + g(Tx),
\label{eq:composite_problem}
\end{equation}

where $f \in C^{1,1}(\mathbb{R}^n)$ is a Lipschitz-smooth function with constant $L_{\nabla f} > 0$, $g: \mathbb{R}^m \to \mathbb{R}$ is a Lipschitz-continuous $\rho$-weakly convex function that may be non-smooth, and $T: \mathbb{R}^n \to \mathbb{R}^m$ is a linear operator.

Examples of problems fitting~\eqref{eq:composite_problem} include the classical LASSO~\cite{lasso} and variants with weakly convex regularizers like the MCP~\cite{mcp} or SCAD penalty~\cite{scad}; low-rank recovery problems such as matrix completion~\cite{candes2008exact} and robust PCA~\cite{candes2011robust}; total variation denoising~\cite{rudin1992nonlinear}; and robust statistical estimation with non-smooth losses such as the Huber loss~\cite{huber1992robust}. The regularizer $g$ is important for ill-posed problems, but it is also what makes \eqref{eq:composite_problem} difficult to optimize, especially when $g$ is non-convex.

We propose \emph{MELMO} (Moreau Envelope Smoothing with Linear Minimization Oracles), which combines two complementary ingredients: Moreau envelope smoothing to produce differentiable surrogates of the composite objective, and linear minimization oracles (LMOs) to generate geometry-aware updates from those surrogates.
 
In the normalized steepest-descent the chosen geometry determines the descent direction, but in many modern applications the ambient Euclidean geometry is not the most informative one. In matrix optimization problems, for example, spectral, nuclear, or entry-wise geometries can induce qualitatively different updates that emphasize different structural biases. For a sparse update, one can use the $\ell_1$ norm for vectors; this is also known as the Gauss-Southwell update \citep{bach2011convex}. The spectral norm for matrices yields full-rank updates; this is also known as spectral descent \citep{carlson2015preconditioned} or Muon \citep{muon}. If the $\ell_2$ norm is used, we recover normalized gradient descent, since the LMO returns $-\varphi_k/\|\varphi_k\|_2$. MELMO is designed to expose this choice explicitly: instead of fixing the Euclidean geometry, it uses an LMO over the unit ball of a chosen norm to produce the corresponding update from the smoothed surrogate, an approach previously explored across different geometries~\cite{balles2020geometry,vasudeva2025generalization,nutini2015coordinate,carlson2015stochastic,kovalev2018stochastic}.

Throughout the remainder of this work, $\|.\|$ denotes the norm associated with the underlying geometry, unless explicitly stated otherwise.

LMOs are used in conditional-gradient or Frank--Wolfe methods, where they generate feasible update directions for constrained problems. In our setting they play a different role. The oracle is not enforcing feasibility in a constraint set; rather, it serves as a geometry-aware subroutine that selects the update direction from the surrogate model relative to the chosen norm.

This perspective sits between two relevant lines of work. On one side, \cite{bohm2021variable} studies variable smoothing using the Moreau envelope for smooth $f$ and weakly convex $g$, using Euclidean gradient steps and proving an $\mathcal O(\epsilon^{-3})$ stationarity complexity, or equivalently an $\mathcal O(k^{-1/3})$ rate on a stationarity measure. On the other side, \cite{yurtsever2019conditional,silveti2026frank} analyze conditional-gradient methods for composite problems in constrained settings, while \cite{boct2020variable,yurtsever2019conditional} obtain $\mathcal O(k^{-1/2})$ functional-gap guarantees in convex regimes with Moreau envelope smoothing. What is missing from this landscape is a method that blends smoothing for weakly convex composite problems (via the Moreau envelope) with the update directions adapted to non-Euclidean geometry through an LMO. MELMO is constructed to fill that gap.

We summarize our contributions are as follows.
\begin{itemize}
    \item We introduce MELMO, an optimization method for unconstrained composite problems like \eqref{eq:composite_problem} that combines the Moreau envelope with exact or inexact LMOs. The method incorporates the geometry of a chosen norm through the LMO, giving structured updates with desirable properties (e.g., sparsity or low rank).
    
    \item For Lipschitz-smooth $f$ and $\rho$-weakly convex $g$, we prove a family of schedule-dependent convergence bounds that makes the trade-off between optimizing the smoothed surrogate and recovering stationarity for the original composite problem explicit. In particular, we isolate a balanced regime with $\mathcal O(k^{-1/4})$ rates for both the smoothed-gradient norm and the composite stationarity proxy, and a smoothed-gradient-favoring regime with $\mathcal O(k^{-1/3})$ for the former together with $\mathcal O(k^{-1/4})$ for the latter.

    \item We develop a variant of MELMO that adjusts the parameters in an epoch-wise fashion, giving stronger theoretical guarantees with computable $\epsilon$-proxy certificates. With power schedules inherited from the main non-convex analysis, this yields an $\mathcal O(\epsilon^{-4})$ certificate complexity, while an epoch-constant specialization with per-epoch restarts improves the bound to $\mathcal O(\epsilon^{-3})$ which is the state-of-the-art bound. 
    Under surjective $T$, these certificates translate into control of the stronger composite stationarity proxy.
    
    \item We evaluate MELMO on sparse low-rank matrix factorization, image-denoising and masked matrix recovery. These experiments illustrate both the practical effect of geometry choice and the current boundary between the theory proved here and broader empirical behavior.
    In particular, even though MELMO does not have an improved theoretical guarantee, it has better practical convergence properties.
\end{itemize}

\subsection{Related Work}\label{subsec:related}

\subsubsection{Non-Euclidean Methods}
Non-Euclidean optimization methods, by which we mean methods that adapt to problem-specific geometries, e.g., through a Riemannian metric \cite{boumal2023introduction}, a Legendre potential function \cite{bregman1967relaxation}, or an LMO \cite{frank1956algorithm,levitin1966constrained}, offer alternatives to the classical Euclidean setting of Hilbert spaces. Conditional Gradient Methods (Frank-Wolfe algorithms)~\cite{frank1956algorithm,levitin1966constrained} use the LMO over convex compact constraint sets, adapting to the geometry of the constraints in a way that avoids projections. Similarly, Nesterov extended these ideas beyond constrained optimization to non-Euclidean norms for unconstrained optimization~\cite{nesterov2009primal}. Later work showed practical benefits: \cite{kelner2014almost} achieved nearly-linear time for max-flow using specific, non-Euclidean norms; \cite{carlson2015preconditioned} used adaptive preconditioning to align descent directions with spectral geometry in deep learning. A study of the advantages of different geometries, highlighting the benefit of sign descent, is given in \cite{balles2020geometry}.

Recent advances build on \cite{carlson2015preconditioned} to apply non-Euclidean principles to neural network training; for an in-depth survey, we defer to \cite{pethick2025training}. Muon~\cite{muon}, and more generally Scion~\cite{scion}, leverage the spectral norm geometry to effectively train neural networks. Extending this, \cite{shulgin2025beyond} considered Scion with inexact LMO computations; in our work, we make the same general assumption about inexactly computing the LMO. Several other works have analyzed the convergence of Frank-Wolfe algorithms under inexact LMO computations, starting with \cite{lacoste2013affine} in the vanilla setting and then \cite{yurtsever2019conditional} and \cite{silveti2021inexact} in the context of non-smooth optimization.

\subsubsection{Non-smooth Optimization}
The use of the Moreau envelope in conjunction with LMOs has been studied in the Frank-Wolfe literature, beginning with \cite{argyriou2014hybrid}, who considered Lipschitz-continuous non-smooth convex regularizers $g\circ T$, and followed by \cite{yurtsever2019conditional}, \cite{silveti2020generalized}, \cite{woodstock2025splitting}, and \cite{silveti2026frank} which generalized this to include constraints and non-convex objectives. We are not aware of any works that study Moreau envelope smoothing with LMOs for unconstrained composite problems like \eqref{eq:composite_problem}.

\section{Notation \& Preliminaries}\label{sec:notnotation_preliminaries}
\subsection{Relevant Function Classes and Regularity}

\begin{definition}[Lipschitz-smooth functions]
A differentiable function $f: \mathbb{R}^n \to \mathbb{R}$ is said to be $L$-smooth with respect to the norm $\|\cdot\|$ if its gradient is $L$-Lipschitz from $\|\cdot\|$ to the dual norm $\|\cdot\|_*$, i.e., there exists $L > 0$ such that
$$
\|\nabla f(x) - \nabla f(y)\|_* \leq L\|x-y\|, \quad \forall x,y \in \mathbb{R}^n.
$$
Equivalently, $f$ satisfies the descent lemma:
$$
f(y) \leq f(x) + \langle \nabla f(x), y-x \rangle + \frac{L}{2}\|y-x\|^2, \quad \forall x,y \in \mathbb{R}^n.
$$
\end{definition}

\begin{definition}[$\rho$-weakly convex functions]
A proper lower semi-continuous function $g: \mathbb{R}^m \to \mathbb{R} \cup \{+\infty\}$ is $\rho$-weakly convex (for $\rho \geq 0$) if the function $x \mapsto g(x) + \frac{\rho}{2}\|x\|^2$ is convex. When $\rho = 0$, we recover standard convex functions.
\end{definition}

Weakly convex functions represent a class that includes many practical non-convex functions used in machine learning and statistics, such as the MCP~\cite{mcp} and the SCAD~\cite{scad} regularizers.
This class serves as a bridge between convex and general non-convex optimization, allowing for the extension of proximal algorithms to non-convex settings while retaining guaranteed convergence to stationary points.

The MCP function with concavity parameter $\mu > 1$ and regularization parameter $\lambda > 0$ is defined as
$$
\text{MCP}_{\mu, \lambda}(x) = 
\begin{cases} 
\lambda |x| - \dfrac{x^2}{2\mu} & \text{if } |x| \leq \mu\lambda, \\[2mm]
\dfrac{1}{2} \mu \lambda^2 & \text{if } |x| > \mu\lambda.
\end{cases}
$$
We remark that MCP is $\mu^{-1}$-weakly convex.

\subsection{The Moreau Envelope and Proximal Operator}

\begin{definition}[Moreau envelope]
For a proper lower semi-continuous function $g: \mathbb{R}^m \to \mathbb{R} \cup \{+\infty\}$ and smoothing parameter $\beta > 0$, the Moreau envelope (or Moreau-Yosida regularization) is defined as:
$$
g^\beta(x) = \min_{u \in \mathbb{R}^m} \left\{ g(u) + \frac{1}{2\beta}\|x - u\|^2_2 \right\}.
$$
\end{definition}

The Moreau envelope provides a smooth approximation to a potentially non-smooth function $g$. When $g$ is $\rho$-weakly convex with $\rho>0$ and $0 < \beta \le 1/(2\rho)$, the Moreau envelope $g^\beta$ is continuously differentiable with with $1/\beta$-Lipschitz gradient (see Lemma~\ref{lem:moreau_grad}); when $\rho=0$, any $\beta>0$ is allowed.

\begin{definition}[Proximal operator]
For a proper lower semi-continuous function $g: \mathbb{R}^m \to \mathbb{R} \cup \{+\infty\}$ and parameter $\beta > 0$, the proximal operator is defined as:
$$
\mathrm{prox}_{\beta g}(x) = \underset{u \in \mathbb{R}^m}{\arg\min} \left\{ \beta g(u) + \frac{1}{2}\|x - u\|^2_2 \right\}.
$$
\end{definition}

Central to first-order optimization is not just smoothness but the accessibility of computing a gradient. For the Moreau envelope, computing a gradient corresponds to computing an associated proximal operator. Indeed, the prox and the Moreau envelope are related, being defined by the same minimization problem. The gradient of the Moreau envelope is given by the following lemma.
\begin{lemma}
\label{lem:moreau_grad}
Let $g: \mathbb{R}^n \to \mathbb{R} \cup \{+\infty\}$ be a proper, lower semi-continuous, $\rho$-weakly convex function. For any $\beta > 0$ satisfying $\beta \le \frac{1}{2\rho}$ for $\rho > 0$, or any $\beta > 0$ otherwise,
the Moreau envelope $g^\beta$ is continuously differentiable and its gradient is given by
\begin{equation*}
\nabla g^\beta(x) = \frac{1}{\beta} \left( x - \prox{\beta g}{x} \right)
\end{equation*}
This gradient is $\frac{1}{\beta}$-Lipschitz continuous.

\end{lemma}
\begin{proof}
    The proof of the convex case ($\rho= 0$) is provided in~\cite{Bauschke2017} (Section 12.4), and for the weakly convex case in~\cite{hoheisel2010proximal} (Corollary 3.4).
\end{proof}
\begin{remark}
    The condition $\beta \le 1/(2\rho)$ is essential for ensuring that the prox is single-valued despite being induced by a non-convex function. In the convex case $\rho=0$, no upper bound on $\beta$ is necessary.
\end{remark}

This property makes the Moreau envelope valuable for non-smooth optimization, as it transforms a potentially non-smooth function into a smooth one while preserving proximity to the original objective function.

\begin{example}
The proximal operator of MCP (also known as \textit{firm thresholding}~\cite{bayram2015convergence}) is
$$
\operatorname{prox}_{\beta \text{MCP}}(x) = 
\begin{cases} 
0 & \text{if } |x| \leq \beta\lambda, \\[2mm]
\dfrac{x - \beta\lambda  \sign(x)}{1 - \beta/\mu} & \text{if } \beta\lambda < |x| \leq \mu\lambda, \\[4mm]
x & \text{if } |x| > \mu\lambda.
\end{cases}
$$
with $0<\beta\leq \mu/2$ interpreted to be the smoothing parameter of the Moreau envelope.
\end{example}

\begin{figure}
    \centering
    \includegraphics[width=0.65\linewidth]{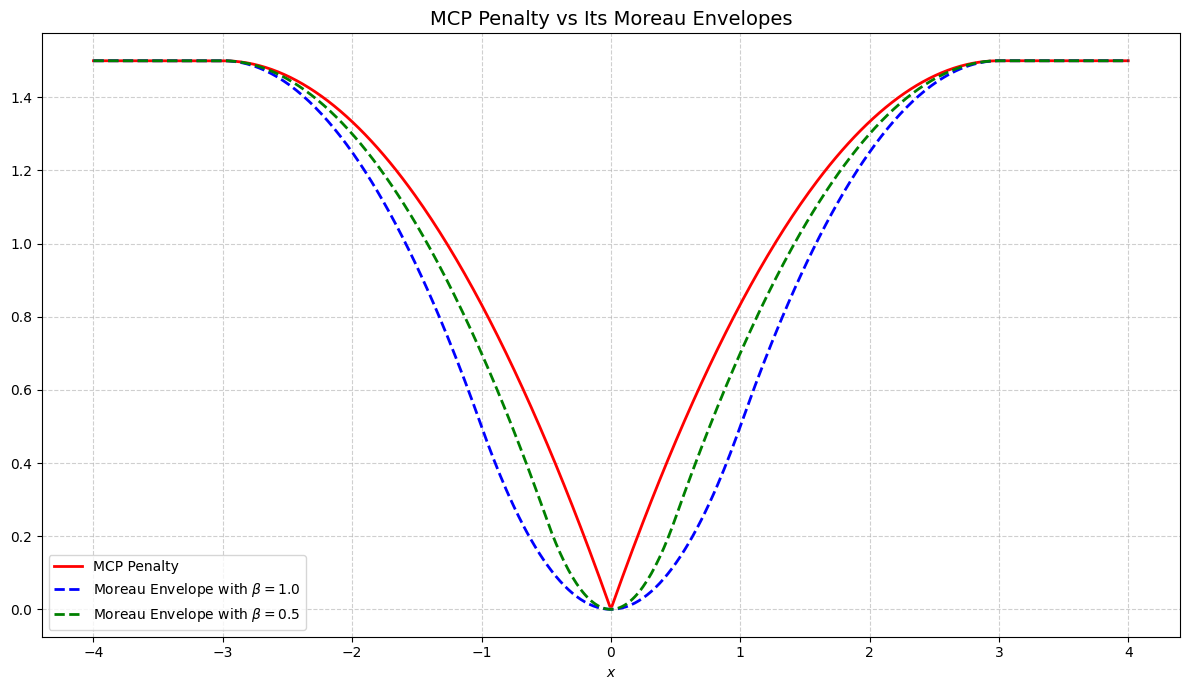}
    \caption{The MCP and its Moreau envelopes with different smoothing parameters $\beta$.}
    \label{fig:envelopes}
\end{figure}

\subsection{Linear Minimization Oracle and Dual Norm}

\begin{definition}[Linear minimization oracle]
For a nonempty convex compact set $\mathcal{D} \subseteq \mathbb{R}^n$ and vector $v \in \mathbb{R}^n$, the linear minimization oracle (LMO) is defined as:
$$
\mathrm{lmo}_{\mathcal{D}}(v) \in \underset{u \in \mathcal{D}}{\arg\min}\ \langle u, v \rangle.
$$
\end{definition}

In our context, we consider the unit ball corresponding to a (possibly non-Euclidean) norm $\mathcal{D} = \{u \in \mathbb{R}^n : \|u\| \leq 1\}$, so:
$$
\mathrm{lmo}(v) \in \underset{\|u\| \leq 1}{\arg\min}\ \langle u, v \rangle,
$$
and any such minimizer satisfies $\langle \mathrm{lmo}(v), v \rangle = -\|v\|_*$ with $\|\mathrm{lmo}(v)\| \leq 1$ where $\|.\|_* = \sup_{\|u\| \leq 1} \langle u, . \rangle$ is the dual norm. In the Euclidean case, this reduces to $\mathrm{lmo}(v) = -v/\|v\|_2$ for $v \neq 0$. Here and throughout, $\|\cdot\|_*$ denotes the dual norm.

In many cases, such as the LMO over the unit ball of the spectral norm, the exact oracle is either unavailable or expensive to compute.

\begin{definition}[$\delta$-inexact LMO]\label{def:inexact-lmo}

Let $\delta\in[0,1)$. We say that an oracle $\widetilde{\mathrm{lmo}}$ is a \emph{$\delta$-inexact LMO} if, for every $v\in\R^n$, there exists an exact solution $d\in \mathrm{lmo}(v)$ such that, with $\widetilde d=\widetilde{\mathrm{lmo}}(v)$,
$$
\|\widetilde d-d\|\le \delta.
$$

\end{definition}

\begin{remark}

The exact case is recovered by taking $\delta=0$. Unless stated otherwise, the algorithms and convergence results below are written for a $\delta$-inexact LMO in the sense of Definition~\ref{def:inexact-lmo}. This definition was also used to analyze the Muon algorithm in \cite{shulgin2025beyond}.

\end{remark}

\subsection{Stationarity Measures}

Because the composite objective $F(x) = f(x) + g(Tx)$ may be non-smooth, we track progress through two related proxies. For $\beta > 0$, define the smoothed objective
\begin{equation}\tag{$\text{P}_\beta$}
F_\beta(x) = f(x) + g^\beta(Tx)
\label{eq:smoothed_composite_problem}
\end{equation}
and the proximal point
$$
z_\beta(x) = \mathrm{prox}_{\beta g}(Tx).
$$

We then aim to prove convergence rates on the quantities
$$
\|\nabla F_\beta(x)\|_*
\qquad\text{and}\qquad
\|Tx-z_\beta(x)\|_2.
$$
The first quantity measures first-order optimality for the smoothed problem, whereas the second measures how faithfully the surrogate model tracks the original composite objective. The role of the smoothing schedule is to balance these two goals: larger $\beta$ makes the smoothed problem easier to optimize, while smaller $\beta$ makes $z_\beta(x)$ closer to $Tx$ and therefore improves fidelity to the original non-smooth objective.

\begin{definition}[$\epsilon$-certificates]

Let $\epsilon>0$. We say that $(x,\beta)$ is an \emph{$\epsilon$-smoothed certificate} if
$$
\|\nabla F_\beta(x)\|_* \le \epsilon.
$$
We say that $(x,z_\beta(x))$ is an \emph{$\epsilon$-proxy certificate} if
$$
\|\nabla F_\beta(x)\|_* \le \epsilon
\qquad\text{and}\qquad
\|Tx-z_\beta(x)\|_2 \le \epsilon.
$$

\end{definition}

\begin{remark}

The algorithmic stopping rules in this paper are stated in terms of $\epsilon$-proxy certificates because both quantities are directly computable.
When $T$ is surjective, an $\epsilon$-proxy certificate yields near-stationarity for the original problem~\eqref{eq:composite_problem}, in the spirit of~\cite{bohm2021variable} (Section 3.2). Let $\sigma_{\min}(T)>0$ denote the smallest singular value of $T$ (positive by surjectivity) and set
$$
\tilde x := x + T^*(TT^*)^{-1}\big(z_\beta(x)-Tx\big),
$$
so that $T\tilde x = z_\beta(x)$ and $\|\tilde x - x\|_2 \le \sigma_{\min}(T)^{-1}\|Tx - z_\beta(x)\|_2$. Since $\nabla g^{\beta}(Tx)\in\partial g(z_\beta(x))$, we get
\begin{equation}\label{eq:surjectiveT}  
\mathrm{dist}_2\big({-\nabla f(\tilde x)},\, T^*\partial g(T\tilde x)\big)
\le
c_1\|\nabla F_\beta(x)\|_* + c_2c_3\,L_{\nabla f}\,\sigma_{\min}(T)^{-1}\|Tx-z_\beta(x)\|_2,
\end{equation}

where $c_1,c_2, c_3>0$ are norm-equivalence constants such that $\|\cdot\|_2\le c_1\|\cdot\|_*$, $\|\cdot\|_2\le c_2\|\cdot\|$ and $\|\cdot\|\le c_3\|\cdot\|_2$. 
In words: $x$ lies within distance $\sigma_{\min}(T)^{-1}\epsilon$ of a point $\tilde x$ at which the composite problem is $\mathcal O(\epsilon)$-stationary. We therefore state convergence rates for the certificates without assuming surjectivity of $T$, and rates on this stronger composite stationarity proxy only when $T$ is assumed to be surjective.

\end{remark}

\section{MELMO}
This algorithm is named \textbf{MELMO} (Moreau Envelope Smoothing with Linear Minimization Oracles) and is shown in Algorithm~\ref{alg:Gamons}. Adopting the framework of normalized steepest-descent methods, MELMO uses a step-size schedule $(\gamma_k)_{k\geq 1}$ and a smoothing schedule $(\beta_k)_{k\geq 1}$. The method can be run with any oracle satisfying Definition~\ref{def:inexact-lmo}.
\begin{algorithm}

\begin{algorithmic}
    \REQUIRE $x_1 \in \R^d$, horizon $K\in\N$, schedules $(\beta_k)_{k=1}^K$ and $(\gamma_k)_{k=1}^K$, and an inexact oracle $\widetilde{\mathrm{lmo}}$ satisfying Definition~\ref{def:inexact-lmo}
    
    \FOR{$k=1,2,\dotsc,K$}
    \STATE $\nabla  F_k(x_k) \leftarrow \nabla f(x_k) + T^*\nabla g^{\beta_k}(Tx_k)$
    \STATE $\widetilde{d}_{k} \leftarrow \widetilde{\mathrm{lmo}}(\nabla  F_k(x_k))$
    \STATE $x_{k+1} \leftarrow x_k + \gamma_k \widetilde{d}_k$
    \ENDFOR
    
    \noindent\RETURN $x_{K+1}$
  \end{algorithmic}
\caption{MELMO}
\label{alg:Gamons}
\end{algorithm}

After computing the gradient of the smoothed function $F_k$, the update direction $\widetilde d_k$ is computed using the LMO for the unit ball associated with the chosen norm. This step encodes the geometry awareness of the algorithm that gives beneficial structure to the updates. In what follows, we analyze the convergence of this algorithm.

\subsection{Convergence Analysis}
Because we optimize a smoothed surrogate of the original objective, we need to control two different quantities: the dual norm of the smoothed gradient and the discrepancy between the proximal point and the original composite model \eqref{eq:composite_problem}. This is exactly where the schedule trade-off enters. Decreasing $\beta_k$ faster improves the fidelity term $\|Tx_k - \mathrm{prox}_{\beta_k g}(Tx_k)\|$, but it also makes the smoothed objective less well-conditioned through the factor $\beta_k^{-1}$ in its smoothness constant.

\begin{theorem}\label{thm:gap original prox}
    Let $g: \mathbb{R}^m \to \mathbb{R}$ be $L_g$-Lipschitz and $\rho$-weakly convex, and $T\colon\mathbb{R}^n\to\mathbb{R}^m$ be a linear operator.
    With $\beta_k = \beta k^{-q}$, where $q\geq 0$ and $\beta\in(0,1/(2\rho)]$ when $\rho>0$ (or any $\beta>0$ when $\rho=0$), we have
    $$
    \|Tx_k - \mathrm{prox}_{\beta_k g}(Tx_k)\|_2 \leq \frac{\beta L_g}{k^{q}}
    $$
\end{theorem}
\begin{proof}
    This is a direct consequence of Lemma 3.3 in~\cite{bohm2021variable}. Since $g$ is $L_g$-Lipschitz, its subgradients at any point are bounded in norm by $L_g$. 
    By the optimality condition defining the proximal point, $\nabla g^{\beta_k}(Tx_k)\in \partial g(\mathrm{prox}_{\beta_k g}(Tx_k))$, so
    $$
    \|Tx_k - \mathrm{prox}_{\beta_k g}(Tx_k)\|_2
    =
    \beta_k \|g^{\beta_k}(Tx_k)\|_2
    \leq \beta_k L_g
    \leq \frac{\beta L_g}{k^{q}}
    $$
\end{proof}

The following theorem summarizes this trade-off for power schedules.

\begin{theorem}[Any-time convergence rate for power schedules]\label{th:1}
    Let $f\in C^{1,1}(\mathbb{R}^n)$ be $L_{\nabla f}$-smooth with respect to $\|\cdot\|$, let $g: \mathbb{R}^m \to \mathbb{R}$ be $L_g$-Lipschitz and $\rho$-weakly convex with $\rho>0$, and let $T\colon\mathbb{R}^n\to\mathbb{R}^m$ be a linear operator. Assume that $F_1^\star := \inf_x F_1(x) > -\infty$ and Algorithm~\ref{alg:Gamons} is run with an oracle satisfying Definition~\ref{def:inexact-lmo} with parameter $\delta\in[0,1)$, and schedules $\beta_k = \beta k^{-q}$, $\gamma_k = \gamma k^{-p}$, where $\gamma>0$, $\beta\in(0, 1/(2\rho)]$, $0<q<p<1$, $p\neq 1/2$, and $2p-q\neq 1$.
    
    Then
    $$
    \min_{j\in[k]}\|\nabla F_j(x_j)\|_*
    \leq
    \mathcal O\!\left(k^{\max\{p-1,\,-p,\,q-p\}}\right).
    $$

    Moreover, if $T$ is surjective, then, with $z_j = \mathrm{prox}_{\beta_j g}(Tx_j)$ and $\tilde x_j := x_j + T^*(TT^*)^{-1}(z_j - Tx_j)$, so that $T\tilde x_j = z_j$ and $\|\tilde x_j - x_j\|_2 \le \sigma_{\min}(T)^{-1}\beta L_g\, j^{-q}$,
    $$
    \min_{j\in[k]}\mathrm{dist}(-\nabla f(\tilde x_j), T^*\partial g(T\tilde x_j))
    \leq
    \mathcal O\!\left(k^{\max\{p-1,\,-p,\,q-p,\,-q\}}\right).
    $$
\end{theorem}
The proof is deferred to Appendix~\ref{proof:1}.

In the certificate terminology introduced above, the first bound shows that among the first $k$ iterates there is an iterate satisfying a smoothed certificate whose accuracy scales like $k^{\max\{p-1,-p,q-p\}}$. Under surjective $T$, the second bound upgrades this to the stronger composite stationarity proxy. The epoch-wise section turns these asymptotic guarantees into explicit computable certificates.

Theorem~\ref{th:1} highlights the role of the schedules $(p,q)$. The parameter $q$ controls how quickly the Moreau envelope tracks the original composite problem through the proximal gap $\|Tx_k-z_k\|=\mathcal O(k^{-q})$, whereas the smoothness of the surrogate $F_k$ scales like $\beta_k^{-1}$ and therefore worsens when $q$ is too large. The parameter $p$ then determines how aggressively we optimize the smoothed objective. Two regimes are particularly relevant in practice.

\begin{corollary}[Balanced regime]\label{cor:balanced-regime}
Under the assumptions of Theorem~\ref{th:1}, if $(p,q) = (7/12,1/3)$, then
$$
\min_{j\in[k]}\|\nabla F_j(x_j)\|_*
\leq \mathcal O(k^{-1/4}),
$$
and, if $T$ is surjective, with $\tilde x_j$ as in theorem~\ref{th:1} we have
$$
\min_{j\in[k]}\mathrm{dist}(-\nabla f(\tilde x_j), T^*\partial g(T\tilde x_j))
\leq \mathcal O(k^{-1/4}).
$$
\end{corollary}

\begin{corollary}[Smoothed-gradient-favoring regime]\label{cor:smoothed-gradient-regime}
Under the assumptions of Theorem~\ref{th:1}, if $(p,q) = (2/3,1/4)$, then
$$
\min_{j\in[k]}\|\nabla F_j(x_j)\|_*
\leq \mathcal O(k^{-1/3}),
$$
and, if $T$ is surjective,with $\tilde x_j$ as in theorem~\ref{th:1} we have
$$
\min_{j\in[k]}\mathrm{dist}(-\nabla f(\tilde x_j), T^*\partial g(T\tilde x_j))
\leq \mathcal O(k^{-1/4}).
$$
\end{corollary}

\begin{corollary}\label{co:1}
    With the same assumptions, but with $g$ convex, Theorem~\ref{th:1} and Corollaries~\ref{cor:balanced-regime} and~\ref{cor:smoothed-gradient-regime} remain valid for any $\beta>0$.
\end{corollary}
The proof is deferred to Appendix~\ref{proof:co1}.

A fixed-horizon choice can also be used to have a better theoretical bound when the horizon $K$ is known in advance.
\begin{theorem}[Horizon-dependent convergence rate]\label{th:1 fixed K}
    We adopt the same assumptions as in Theorem~\ref{th:1} on $f$, $g$, $T$, and the oracle, but run Algorithm~\ref{alg:Gamons} with the constant schedules $\beta_k = \beta K^{-1/3}$, $\gamma_k = \gamma K^{-2/3}$, where $\gamma>0$, $\beta\in(0, 1/(2\rho)]$ and $K\in \mathbb{N}^*$ a fixed horizon.
    Then
                $$
    \min_{j\in[K]}
    \|\nabla F_j(x_j)\|_*
        \leq 
        \mathcal{O}(K^{-1/3})
    $$
    Moreover, if $T$ is surjective with $\tilde x_j$ as in theorem~\ref{th:1} we have
    $$
\min_{j\in[K]}\mathrm{dist}(-\nabla f(\tilde x_j), T^*\partial g(T\tilde x_j))
    \leq
        \mathcal{O}(K^{-1/3})
    $$
\end{theorem}
The proof is deferred to Appendix~\ref{proof:1 fixed K}.

\subsection{Epoch-wise certification}
The bounds proven in the previous section control the minimum smoothed-gradient norm over a long prefix of iterates. Therefore it is possible that each certificate is minimized for a different index. For certification, however, we also need the proximal gap $\|Tx_j-z_j\|_2$ to be small at a comparable index. If the smallest smoothed-gradient norm is attained too early, these two requirements may not line up. The purpose of the epoch-wise variant is therefore to produce a computable $\epsilon$-proxy certificate, that is, an iterate for which both $\|\nabla F_j(x_j)\|_*$ and $\|Tx_j-z_j\|_2$ are below a prescribed tolerance.

To ensure this, we propose the \textit{epoch-wise} version of MELMO in Algorithm~\ref{alg:epoch_vs}. The algorithm itself is written for generic schedules $(\beta_k,\gamma_k)$; below we analyze two different instantiations. The first keeps the power schedules from the main theorem and yields an $\mathcal O(\epsilon^{-4})$ certificate complexity. The second specializes the schedules to be constant within each epoch and, together with a restart at the start of each epoch from a point whose (computable) objective value is no worse than the initial one, improves this to $\mathcal O(\epsilon^{-3})$.

\begin{algorithm}[ht!]
  \begin{algorithmic}
    \REQUIRE $x_1 \in \R^d$, schedules $(\beta_k)_{k\ge 1}$ and $(\gamma_k)_{k\ge 1}$, tolerance $\epsilon>0$, and an inexact oracle $\widetilde{\mathrm{lmo}}$ satisfying Definition~\ref{def:inexact-lmo}
    \FOR{$l=0,1,\dotsc$}
    \STATE $S_l \leftarrow \infty$, $j_l \leftarrow 2^l$
    \FOR{$k=2^l,2^l+1,\dotsc,2^{l+1}-1$}
    \STATE $\nabla  F_k(x_k) \leftarrow \nabla f(x_k) + T^*\nabla g^{\beta_k}(Tx_k)$
    \STATE $\widetilde{d}_{k} \leftarrow \widetilde{\mathrm{lmo}}(\nabla  F_k(x_k))$
    \STATE $x_{k+1} \leftarrow x_k + \gamma_k \widetilde{d}_k$
    \IF{$ \Vert \nabla F_{k+1}(x_{k+1})\Vert_* \le S_l$}
    \STATE $S_l \leftarrow \Vert \nabla F_{k+1}(x_{k+1})\Vert_*$
    \STATE $j_l \leftarrow k+1$
    \STATE $z_{k+1} \leftarrow \prox{\beta_{k+1} g}{Tx_{k+1}}$
    \IF{$S_l \le \epsilon$ and $ \Vert Tx_{k+1} - z_{k+1} \Vert_2 \le \epsilon$}
    \STATE \RETURN $x_{k+1}$
    \ENDIF
    \ENDIF
    \ENDFOR
    \ENDFOR
  \end{algorithmic}
  \caption{Epoch-wise MELMO}%
  \label{alg:epoch_vs}%
\end{algorithm}

For each epoch $l$, Algorithm~\ref{alg:epoch_vs} keeps track of the lowest value of the smoothed function's gradient norm $S_l$ within that epoch, together with the corresponding iterate $x_{j_l}$. Whenever the currently best iterate in the epoch satisfies both parts of the $\epsilon$-proxy certificate, the method stops and returns that certified point.

\begin{theorem}[Power-schedule epoch certificate]
\label{th:Epoch-wise}
    Let the assumptions of Theorem~\ref{th:1} hold. If Algorithm~\ref{alg:epoch_vs} is run with the power schedules
    $$
    \beta_k = \beta k^{-1/4}
    \qquad\text{and}\qquad
    \gamma_k = \gamma k^{-2/3},
    $$
    then after $\mathcal O(\epsilon^{-4})$ iterations it returns an iterate $x_{j_l}$ such that, with $z_{j_l} = \mathrm{prox}_{\beta_{j_l} g}(Tx_{j_l})$, the pair $(x_{j_l},z_{j_l})$ is an $\epsilon$-proxy certificate.
\end{theorem}
The proof is deferred to Appendix~\ref{proof:Epoch-wise}.

\begin{remark}

If, in addition, $T$ is surjective, the translation recalled in the stationarity subsection turns the output of Theorem~\ref{th:Epoch-wise} into near-stationarity at the point $\tilde x_{j_l} = x_{j_l} + T^*(TT^*)^{-1}(z_{j_l}-Tx_{j_l})$: both $\mathrm{dist}(-\nabla f(\tilde x_{j_l}),T^*\partial g(T\tilde x_{j_l}))$ and $\|\tilde x_{j_l}-x_{j_l}\|_2$ are $\mathcal O(\epsilon)$. Note that here the two certificate quantities are controlled at the same index $j_l$ by construction, so no further alignment argument is needed.
\end{remark}

\begin{proposition}[Epoch-constant schedule specialization with restarts]
\label{prop:epoch-wise-fast}
Under the assumptions of Theorem~\ref{th:1}, consider the specialization of Algorithm~\ref{alg:epoch_vs} in which, throughout epoch $l$, one uses
$$
\beta_k = \beta 2^{-(l+1)/3}
\qquad\text{and}\qquad
\gamma_k = \gamma 2^{-2(l+1)/3}.
$$
All certificate evaluations inside epoch $l$ (the gradient norms and proximal gaps in the stopping rule) are performed with the smoothing parameter $\beta 2^{-(l+1)/3}$ of the current epoch, and, at the start of each epoch, the iterate is reset to a point $x^0_l$ satisfying $F(x^0_l)\le F(x_1)$ --- for instance $x^0_l = x_1$, or the iterate with the smallest composite objective value $F$ observed so far\footnote{Note that $F = f + g\circ T$ is directly computable.}.
Then, after $\mathcal O(\epsilon^{-3})$ iterations the algorithm returns an iterate $x_{j_l}$ such that, with $z_{j_l} = \mathrm{prox}_{\beta_{j_l} g}(Tx_{j_l})$, the pair $(x_{j_l},z_{j_l})$ is an $\epsilon$-proxy certificate.

\end{proposition}

\begin{remark}
If $T$ is surjective, Proposition~\ref{prop:epoch-wise-fast} likewise yields control of the composite stationarity proxy at the returned iterate.
\end{remark}

\section{Experiments}\label{sec:exp}

We evaluate MELMO on three tasks: sparse low-rank matrix factorization, total variation image denoising and masked matrix recovery\footnote{Code and datasets are available at \url{https://github.com/Farid-Najar/MELMO}.}. All runs are deterministic (no stochastic gradients) and use a single fixed set of method-level hyperparameters per algorithm, listed in Appendix Tables~\ref{tab:hyperparameters_separate}, \ref{tab:hyperparameters_mask} and~\ref{tab:hyperparameters_denoising}; no per-dataset tuning is performed. Each method is run from the same random initialization for every (dataset, rank) pair. For the low-rank experiments, Table~\ref{tab:performance} reports the final penalized objective value~\eqref{eq:pen_general_problem_2factors}, while Appendix Table~\ref{tab:reconstruction loss} reports the reconstruction loss $\|Y-WH\|_F^2$ separately, so that the effect of the regularizer can be distinguished from approximation quality.

\subsection{Sparse low-rank matrix factorization}\label{subsec:low_rank}
\begin{table*}[ht!]
\centering
\caption{Performance comparison of the algorithms across five datasets, each with three ranks. Values represent the final penalized objective $F$.}
\label{tab:performance}
{\small
\setlength{\tabcolsep}{1pt}
\begin{tabular}{l l c c c c}
\toprule
\textbf{Dataset} & \textbf{rank} & \shortstack{\textbf{Sub-}\\\textbf{gradient}} & \shortstack{\textbf{Variable}\\\textbf{smoothing}} & \shortstack{\textbf{MELMO}\\$(p = 7/12, q = 1/3)$} & \shortstack{\textbf{MELMO}\\$(p = 2/3, q = 1/4)$}\\
\midrule

\multirow{3}{*}{Camera} 
& 10 
& $521.9$ 
& $448.5$ 
& $\textbf{390.4}$
& $441.6$
\\

& 20 
& $434.6$ 
& $383.9$ 
& $\textbf{272.3}$
& $347.8$
\\

& 30 
& $406.8$ 
& $377.8$ 
& $\textbf{238.4}$
& $322.8$
\\
\midrule

\multirow{3}{*}{Spectrometer} 
& 10 
& $85.2$ 
& $83.6$ 
& $\textbf{52.2}$
& $67.5$
\\

& 20 
& $87.4$
& $78.0$
& $\textbf{64.3}$
& $108.1$
\\

& 30 
& $\textbf{87.7}$ 
& $87.9$ 
& $89.7$
& $113.3$
\\

\midrule

\multirow{3}{*}{Olivetti}
& 10 
& $1.178\times10^4$ 
& $1.206\times10^4$ 
& $\textbf{1.143}\times10^4$
& $1.287\times10^5$
\\

& 20 
& $\textbf{8.508}\times10^3$ 
& $1.074\times10^4$ 
& $8.621\times10^3$
& $1.211\times10^5$
\\

& 30 
& $\textbf{6.914}\times10^3$ 
& $1.014\times10^4$ 
& $7.516\times10^3$
& $9.376\times10^4$
\\
\midrule

\multirow{3}{*}{Football} 
& 10 
& $620.7$ 
& $572.3$ 
& $\textbf{569.5}$
& $573.1$
\\

& 20 
& $495.2$ 
& $400.5$ 
& $\textbf{395.5}$
& $401.3$
\\

& 30 
& $420.8$ 
& $289.9$ 
& $\textbf{283.5}$
& $291.1$
\\
\midrule

\multirow{3}{*}{Low rank synthetic} 
& 10 
& $72.7$ 
& $87.6$ 
& $48.1$
& $\textbf{47.9}$
\\

& 20 
& $71.1$ 
& $87.2$ 
& $\textbf{49.1}$
& $50.8$
\\

& 30 
& $88.5$ 
& $101.3$ 
& $70.1$
& $\textbf{67.5}$
\\
\bottomrule
\end{tabular}
}
\end{table*}

The general low-rank decomposition problem consists of approximating a given matrix$Y \in \mathbb{R}^{m \times n}$ by a matrix $X$ of rank $r < \min(m, n)$. To find the best such approximation, it is common to consider the following optimization problem:    
\begin{equation}
\min_{X} \| Y - X \|_F^2 \quad \text{s.t.} \quad \text{rank}(X) \leq r,
\label{eq:general_problem}
\end{equation}
where $\| . \|_F$ denotes the Frobenius norm.
To enforce the rank constraint on $X$, a commonly used approach is to write the matrix as $X = WH$ with $W \in \mathbb{R}^{m \times r}$ and $H \in \mathbb{R}^{r \times n}$. Thus, the optimization problem \eqref{eq:general_problem} can be reformulated as follows: 
\begin{equation}
    \label{eq:general_problem_2factors}
    \min_{W, H}{\| Y - WH \|_F^2} . 
\end{equation} 
For sufficiently small values of $r$, specifically $r < \frac{mn}{m+n}$, this decomposition enables compression of the original matrix. 
Beyond compression, low-rank factorizations are valuable for their ability to extract meaningful features from data, which has established them as essential techniques in data analysis and machine learning. 

To enforce additional properties (e.g., sparsity), we can add a regularization function $g$, thus recovering the optimization problem~\eqref{eq:composite_problem}.

For instance, we consider the regularizers of the form
$$
g(X) = \sum_{i=1}^m\sum_{j=1}^n \phi(X_{i, j})
$$
for a matrix $X\in\R^{m\times n}$, where $\phi$ is a non-smooth function such as the absolute value, MCP, or SCAD. Note that the linear operator $T$ is the identity here.
In this work, we take $\phi(\cdot) = \text{MCP}_{\mu, \lambda}(\cdot)$ with parameters $\mu = 3$ and $\lambda = 1$.
Thus, our concrete composite minimization problem is
\begin{equation}
    \label{eq:pen_general_problem_2factors}
    \min_{W, H}{\| Y - WH \|_F^2} 
    + 
    \sum_{i=1}^m\sum_{j=1}^r
    \text{MCP}_{\mu, \lambda}(W_{i, j})
\end{equation} 
The regularization is applied to $W$; applying it to $H$ instead is equally valid. The goal is to promote sparsity in one of the two factors, improving interpretability without substantially degrading the factorization quality.

We use the benchmark instances of~\cite{wertz2025efficient}, comprising four real matrices (Camera, Spectrometer, Olivetti, Football) and one synthetic low-rank matrix.
We compute the direction with respect to the $\ell_2 \to \ell_2$ operator norm, i.e., the spectral norm used by the Muon optimizer~\cite{muon}.
Choosing this geometry for the descent enforces \textit{dense} updates, meaning the resulting direction has all of its singular values equal to one.
This geometry is particularly well suited for neural networks, where the raw gradients are often low-rank matrices  and concentrated along a few directions~\cite{muon}.
Given an SVD $X = U\Sigma V^T$ of a matrix $X\in\R^{m\times n}$, the LMO with respect to the unit ball $\mathcal D_{\mathcal S_\infty}$ associated with this norm is
$$
\mathrm{lmo}_{\mathcal D_{\mathcal S_\infty}}(X) = -UV^T
$$
This LMO can be computed without explicitly performing an SVD by leveraging Newton--Schulz iterations that approximate $UV^T$. In the experiments reported here, we normalize the objective matrices $Y$ with Min--Max scaling before the start of the experiments, which stabilizes the approximate LMO without changing the ordering of its entries. 

In this study, we use a 6-iterations Newton--Schulz method (see the code for more details). One can reduce the number of iterations in order to speed up the computation, but this may have an important impact on the global performance of the algorithm as shown in Figure~\ref{fig:impact_of_iterations}.

We evaluate two variants of MELMO corresponding to the schedule pairs $(p, q) = (7/12, 1/3)$ and $(p, q) = (2/3, 1/4)$, with the remaining hyperparameters given in Table~\ref{tab:hyperparameters_separate}. These are compared against the variable smoothing algorithm of~\cite{bohm2021variable} and a sub-gradient method with fixed step size. The comparison is designed to assess whether geometry-aware smoothing is competitive with established Euclidean baselines under a uniform hyperparameter policy, i.e., the same settings for each method across all datasets.

Full results are presented in Table~\ref{tab:performance}. The balanced schedule $(p,q)=(7/12,1/3)$ achieves the lowest penalized objective on Camera (all ranks), Football (all ranks), and the synthetic dataset (ranks~10 and~20), and is competitive on Spectrometer at ranks~10 and~20. However, on Spectrometer at rank~30 the sub-gradient method obtains a lower objective ($87.7$ vs.\ $89.7$), and on Olivetti at ranks~20 and~30 both MELMO variants are outperformed by the sub-gradient baseline. The more aggressive schedule $(p,q)=(2/3,1/4)$ exhibits a pronounced failure mode on Olivetti, where the penalized objective is an order of magnitude larger than the competing methods; this suggests that the faster smoothing decay overwhelms the optimization on this dataset. Overall, the $(7/12,1/3)$ schedule is the more robust of the two. Appendix Table~\ref{tab:reconstruction loss} reports the reconstruction loss $\|Y-WH\|_F^2$ separately: these values largely track the penalized-objective ranking, confirming that the objective improvements translate into comparable or better factorization quality, with Olivetti again exposing the main instability. The convergence trajectories on Camera (Figure~\ref{fig:camera_Fs}) show that the $(7/12,1/3)$ schedule reaches lower objective values earlier than the baselines on this particular dataset.

\begin{figure}[ht!]
    \centering
    \begin{minipage}{0.49\textwidth}
        \centering
        \includegraphics[width=\linewidth]{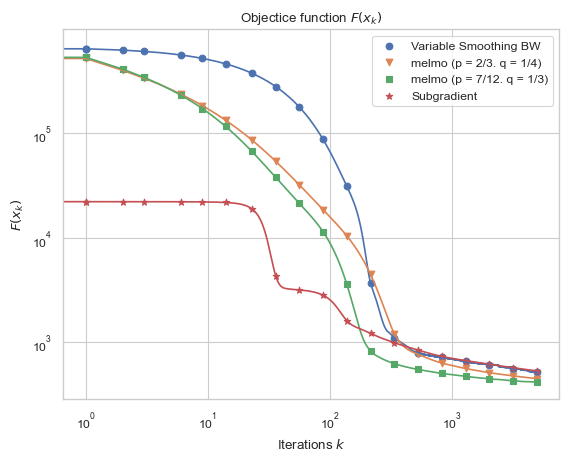}\\
        ($r = 10$)
    \end{minipage}
    \hfill
    \begin{minipage}{0.49\textwidth}
        \centering
        \includegraphics[width=\linewidth]{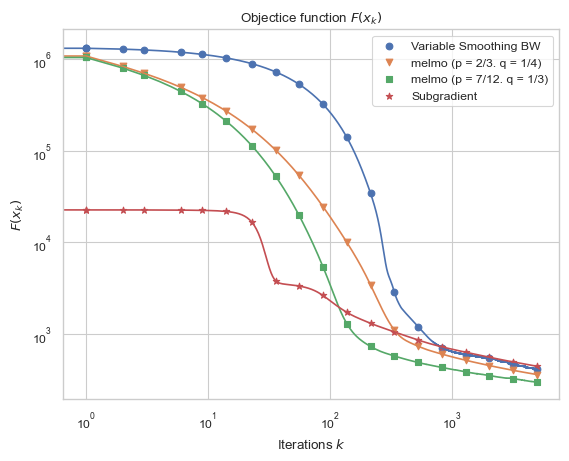}\\
        ($r = 20$)
    \end{minipage}
    \hfill
    \begin{minipage}{0.49\textwidth}
        \centering
        \includegraphics[width=\linewidth]{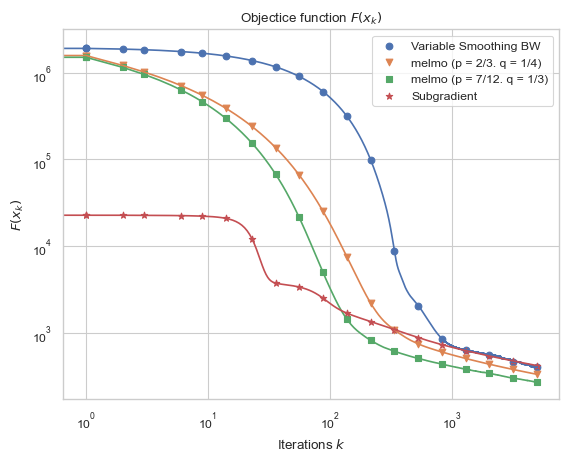}\\
        ($r = 30$)
    \end{minipage}
    \caption{Penalized objective $F$ vs.\ iteration (log scale) on the Camera dataset for ranks $r\in\{10,20,30\}$.}
    \label{fig:camera_Fs}
\end{figure}

\subsection{Total variation denoising}

Given an image $y \in \mathbb{R}^{m\times n}$, we want to recover a denoised version of $y$ using the anisotropic model
\begin{equation*}
    \min\limits_{W\in\mathbb{R}^{m\times n}} \frac{1}{2}\|W-Y\|_F^2
    +
    \sum_{i,j}\text{MCP}_{\mu,\lambda}\!\left((\nabla_x W)_{i,j}\right)
    +
    \sum_{i,j}\text{MCP}_{\mu,\lambda}\!\left((\nabla_y W)_{i,j}\right),
\end{equation*}
where $\nabla W = (\nabla_x W,\nabla_y W)$ denotes the discrete forward-difference gradient along rows and columns.

The adjoint operator $\nabla^*$ acts on a pair of matrices $(P_x, P_y) \in \mathbb{R}^{m \times n} \times \mathbb{R}^{m \times n}$ and is given by the negative discrete divergence. With zero boundary conditions (i.e., values outside the domain are treated as zero), it is expressed as:

$$
(\nabla^* P)_{i,j} = 
- \Big[ (P_x)_{i,j} - (P_x)_{i,j-1} \Big] 
- \Big[ (P_y)_{i,j} - (P_y)_{i-1,j} \Big],
$$
for $i = 1,\dots,m,\; j = 1,\dots,n$, where we adopt the convention:
$$
(P_x)_{i,0} = 0 \quad \text{and} \quad (P_y)_{0,j} = 0.
$$

Equivalently, in piecewise form:

$$
(\nabla^* P)_{i,j} =
\begin{cases}
- (P_x)_{i,1} - (P_y)_{1,j}, & i=1,\; j=1, \\
- \big[(P_x)_{i,j} - (P_x)_{i,j-1}\big] - (P_y)_{1,j}, & i=1,\; j>1, \\
- (P_x)_{i,1} - \big[(P_y)_{i,j} - (P_y)_{i-1,j}\big], & i>1,\; j=1, \\
- \big[(P_x)_{i,j} - (P_x)_{i,j-1}\big] - \big[(P_y)_{i,j} - (P_y)_{i-1,j}\big], & i>1,\; j>1.
\end{cases}
$$

In compact notation, this is often written as:

$$
\nabla^* = -\operatorname{div}.
$$

The smoothed gradient at each iteration is
\begin{equation*}
    \nabla F_j(W) = x-y + \frac{\nabla^\ast(\nabla W - \mathrm{prox}_{\beta_j \text{MCP}}(\nabla W))}{\beta_j}
\end{equation*}
where the proximal operator acts elementwise on the horizontal and vertical components of $\nabla W$.

This denoising experiment should be read as an empirical extension of the method. The non-convex convergence theory requires surjectivity of $T$ only when translating the smoothed-gradient bound into the stronger composite stationarity proxy; the discrete image gradient $\nabla\colon\R^{m\times n}\to\R^{m\times n}\times\R^{m\times n}$ does not satisfy this assumption. We therefore use this section to assess the practical behavior of geometry-aware smoothing rather than to claim a direct instantiation of the surjective-$T$ theorem.

We apply MELMO to the Camera image with the same MCP parameters ($\mu=3$, $\lambda=1$) and the schedule $(p,q)=(7/12,1/3)$. We compare against variable smoothing and the sub-gradient method, and we also compare three LMO geometries for MELMO: spectral, $\ell_2$, and nuclear (see~\cite{scion} for their expressions). 

\begin{figure}[ht!]
    \centering
    \includegraphics[width=.75\linewidth]{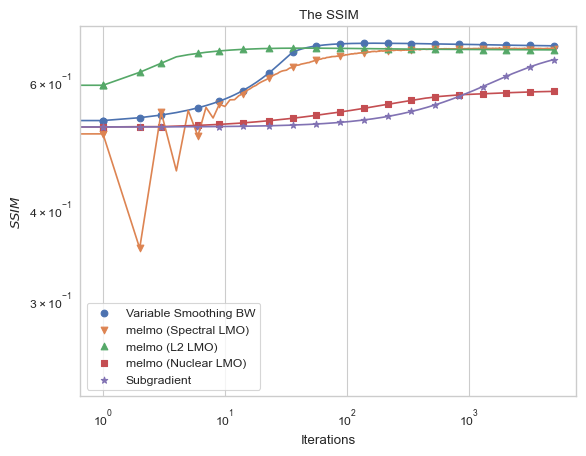}
    \caption{SSIM vs.\ iteration (log scale) for total variation denoising on the Camera image, comparing three MELMO geometries against variable smoothing and sub-gradient baselines.}
    \label{fig:camera_denoising}
\end{figure}

Figure~\ref{fig:camera_denoising} shows the SSIM~\cite{wang2004image}convergence curves. All three MELMO geometries reach higher SSIM values than the sub-gradient baseline on this test image. Among the LMOs, the $\ell_2$ oracle performs best, the spectral oracle is comparable to variable smoothing, and the nuclear oracle is the weakest of the three geometries. Because only a single image and noise level are used, these observations should be taken as illustrative rather than as a general ranking of geometries for denoising.

We also compare the regular MELMO and the epoch-wise MELMO under the same denoising setting with the spectral LMO and $(p, q) = (2/3,1/3)$. Figure~\ref{fig:epoch} shows both methods have comparable performance.
Moreover, as shown in Figure~\ref{fig:proximalGap_epochs}, the proximal gap goes down faster with the epoch-wise MELMO than the regular variant, so, the first has a better fidelity towards the original problem than the latter.
We also see that the dual norm of the gradients $\|\nabla F_k\|_*$ oscillates for both methods but it goes down at a comparable rate.

\begin{figure}[ht!]
    \centering
    \includegraphics[width=0.49\linewidth]{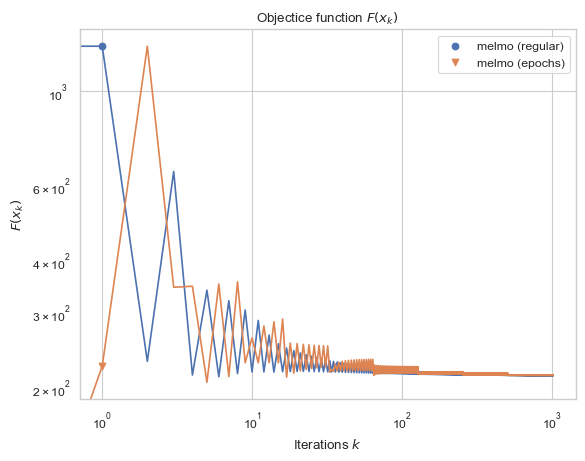}
    \includegraphics[width=0.49\linewidth]{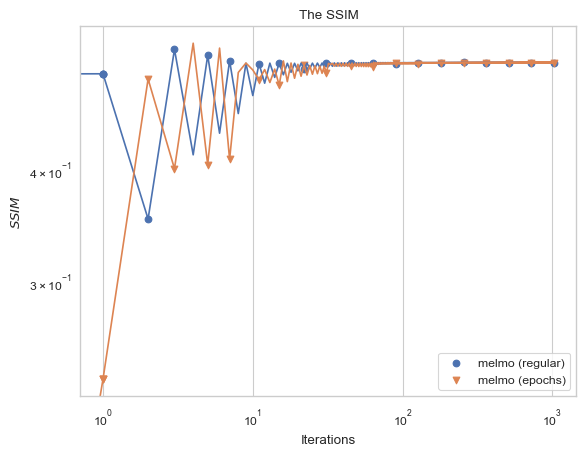}
    \caption{Objective value and SSIM comparison between regular MELMO and epoch-wise MELMO with the spectral LMO on the Camera denoising problem.}
    \label{fig:epoch}
\end{figure}

\begin{figure}[ht!]
    \centering
    \includegraphics[width=0.49\linewidth]{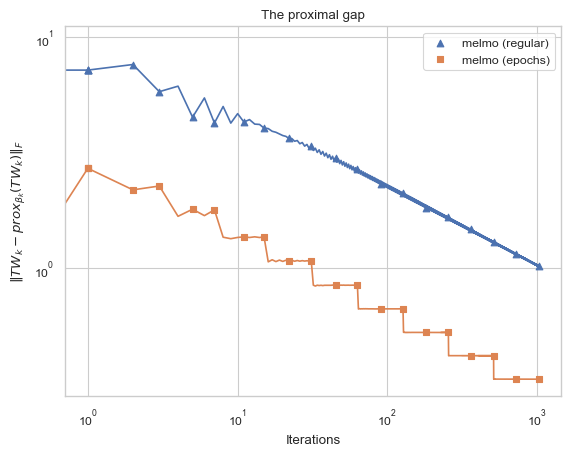}
    \includegraphics[width=0.49\linewidth]{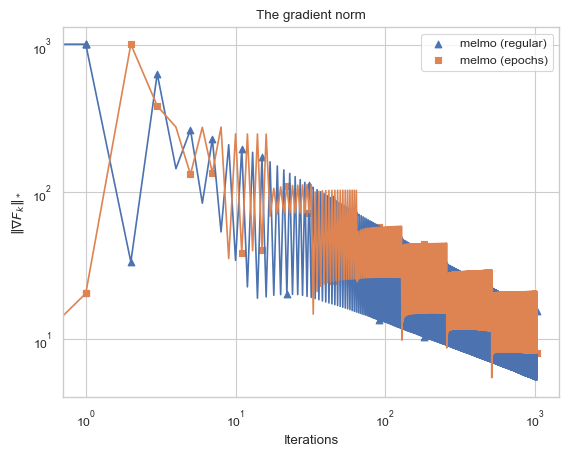}
    \caption{The proximal gap or fidelity term ($\|TW_k - \prox{\beta_k}{TW_k}\|$), and the dual norm of the gradients $\|\nabla F_k\|_*$ comparison between regular MELMO and epoch-wise MELMO with the spectral LMO on the Camera denoising problem.}
    \label{fig:proximalGap_epochs}
\end{figure}

\subsection{Masked matrix recovery}
\label{subsec:masked_matrix_recovery}

The sparse low-rank factorization experiment in Section~\ref{subsec:low_rank} is computationally natural, but its smooth part
$$
(W,H)\mapsto \|Y-WH\|_F^2
$$
does not have a globally Lipschitz gradient. To evaluate MELMO on a matrix problem that satisfies the theoretical assumptions globally while retaining the matrix geometry of interest, we introduce a masked matrix recovery experiment.

Let $Y\in\mathbb{R}^{m\times n}$ be a reference matrix, and let $\Omega\subseteq \{1,\dots,m\}\times\{1,\dots,n\}$ denote the set of observed entries. Define the binary mask $M\in\{0,1\}^{m\times n}$ by
$$
M_{ij}
=
\begin{cases}
1, & (i,j)\in\Omega,\\
0, & (i,j)\notin\Omega,
\end{cases}
$$
and write $P_\Omega(X)=M\odot X$, where $\odot$ denotes entrywise multiplication.
We consider the masked matrix recovery problem
\begin{equation}
\label{eq:masked_matrix_recovery}
\min_{X\in\mathbb{R}^{m\times n}}
\frac12\|P_\Omega(X-Y)\|_F^2
+
\alpha
\sum_{i=1}^{m}\sum_{j=1}^{n}
\operatorname{MCP}_{\mu,\lambda}(X_{ij}),
\end{equation}
where $\alpha>0$ is a regularization weight and $\operatorname{MCP}_{\mu,\lambda}$ is the MCP penalty defined in Section~\ref{sec:notnotation_preliminaries}.
The mask $M$ is fixed during optimization and is not a decision variable.

Problem~\eqref{eq:masked_matrix_recovery} is an instance of the composite problem
$(P)$ with $T=I$. Identifying $\mathbb{R}^{m\times n}$ with
$\mathbb{R}^{mn}$, we may write
$$
f(X)=\frac12\|P_\Omega(X-Y)\|_F^2,
\qquad
g(X)=
\alpha
\sum_{i=1}^{m}\sum_{j=1}^{n}
\operatorname{MCP}_{\mu,\lambda}(X_{ij}).
$$  
The smooth part is quadratic, with gradient
$$
\nabla f(X)=P_\Omega(X-Y).
$$
Hence, for any $X,X'\in\mathbb{R}^{m\times n}$,
$$
\|\nabla f(X)-\nabla f(X')\|_F
=
\|P_\Omega(X-X')\|_F
\le
\|X-X'\|_F,
$$
so $f$ is globally $1$-smooth with respect to the Frobenius norm. If MELMO is implemented with another primal norm, such as the spectral norm, the smoothness assumption still holds with a finite norm-equivalence constant.
The regularizer $g$ is separable, Lipschitz-continuous, and
$\mu^{-1}$-weakly convex. Moreover, $T=I$ is surjective and
$\sigma_{\min}(T)=1$. Therefore, the convergence certificates of
theorems of the previous section apply directly to problem~\eqref{eq:masked_matrix_recovery}. In particular, an $\epsilon$-proxy certificate
$$
\|\nabla F_{\beta_j}(X_j)\|_*\le \epsilon,
\qquad
\|X_j-\operatorname{prox}_{\beta_j g}(X_j)\|_F\le \epsilon
$$
implies near-stationarity for the original composite problem without requiring an additional surjectivity argument.

This experiment is a standard matrix completion or missing-data recovery model. The observed-entry term enforces fidelity to the known entries of $Y$, while the MCP regularizer encodes structural bias such as sparsity or, in spectral variants, low-rank structure. It is relevant, for example, to recommender systems, recovering sparse images, image inpainting, sensor-network imputation, and low-rank data recovery.

Unless stated otherwise, we generate $\Omega$ by sampling exactly $\lfloor \omega mn\rfloor$ entries uniformly at random without replacement, where $\omega\in(0,1]$ is the observation ratio. The default choice is $\omega=1/2$. We also report results for $\omega\in\{0.3,0.5,0.7\}$ to assess sensitivity to the amount of missing data. The same mask $M$ is used for all compared algorithms.

For these experiments, we compare multiple variants of MELMO alongside the baselines in Section~\ref{subsec:low_rank}, i.e. the variable smoothing baseline of~\cite{bohm2021variable} and the fixed-step subgradient method.
First, we choose MELMO with
$$
(p,q)=\left(\frac{7}{12},\frac13\right)
$$
corresponding to the balanced regime which showed the best performance in low-rank factorization experiments(see Table~\ref{tab:performance}). This regular variant of MELMO will be studied under two LMOs : the $\ell_2$ LMO and the spectral LMO computed approximately using Newton-Schulz iterations, as in Section~\ref{subsec:low_rank}.
Moreover, we compare these regular MELMO variants with epoch-wise MELMO variant with $(p,q) = (2/3, 1/3)$ that also uses the spectral LMO.
The proximal operator of $g$ is computed entrywise by firm thresholding.
In this work, we set $\alpha=1$.

Let $X_k$ denote the iterate produced by a given method at iteration $k$, and let $P_{\Omega^c}(X)=(1-M)\odot X$ denote the projection onto the unobserved entries. We report the following quantities:
\begin{align}
\text{penalized objective:}\qquad
&F(X_k)
=
\frac12\|P_\Omega(X_k-Y)\|_F^2
+
g(X_k),
\\
\text{observed residual:}\qquad
&\|P_\Omega(X_k-Y)\|_F^2,
\\
\text{relative held-out error:}\qquad
&
\frac{\|P_{\Omega^c}(X_k-Y)\|_F}
{\|P_{\Omega^c}(Y)\|_F},
\\
\text{smoothed certificate:}\qquad
&\|\nabla F_{\beta_k}(X_k)\|_*,
\\
\text{proximal gap:}\qquad
&\|X_k-\operatorname{prox}_{\beta_k g}(X_k)\|_F.
\end{align}
The first four metrics measure statistical recovery quality, while the last two measure compliance with the theoretical stationarity certificates.

\paragraph{Results.}

To maintain readability, we present a representative subset of figures illustrating the core performance trends of all algorithms. The complete set of experimental results, including additional plots in different settings, is deferred to the appendix.

As shown in Table~\ref{tab:performanceM}, MELMO with $\ell_2$ LMO outperforms with a slight margin over all other algorithms. Meanwhile, the other MELMO variants remain competitive.
That said, Figure~\ref{fig:masked_objectives_sample} shows that MELMO with the spectral LMO—both the standard and epoch-wise variants—reduces the objective value faster than all competing methods in early iterations.
Figure~\ref{fig:masked_rLoss_sample} shows that all methods reach comparable levels of the observed residual loss in the long run, while the spectral LMO variants of MELMO achieve faster convergence to these values.
Overall, the subgradient method performs poorly across all these experimental metrics. We tested a range of alternative step sizes during development, but none resolved the underlying convergence issues, and the method's behavior remained consistently weak.

\begin{table*}[ht!]
\centering
\caption{Performance comparison of the algorithms across 2 datasets, each with three masking probabilities. Values represent the final penalized objective $F$.}
\label{tab:performanceM}
{\small
\setlength{\tabcolsep}{1pt}
\begin{tabular}{l l c c c c c}
\toprule
\textbf{Dataset} & $\omega$\quad & \shortstack{\textbf{Sub-}\\\textbf{gradient}} \quad& \shortstack{\textbf{Variable}\\\textbf{smoothing}} \quad& \shortstack{\textbf{MELMO}\\$(\ell_2)$} \quad&
\shortstack{\textbf{MELMO}\\\textbf{Epoch-wise}} \quad& \shortstack{\textbf{MELMO}\\$(spectral)$}\\
\midrule

\multirow{3}{*}{Football} 
& $0.3$ 
& $386.3$ 
& $3.363$ 
& $\textbf{3.262}$
& $3.403$ 
& $3.513$
\\

& $0.5$ 
& $511.7$ 
& $5.457$ 
& $\textbf{5.394}$
& $5.608$ 
& $5.789$
\\

& $0.7$ 
& $645.2$ 
& $7.549$ 
& $\textbf{7.442}$
& $7.631$ 
& $7.821$
\\
\midrule

\multirow{3}{*}{Misérables} 
& $0.3$ 
& $151.0$ 
& $1.391$ 
& $\textbf{1.355}$
& $1.390$ 
& $1.458$
\\

& $0.5$ 
& $221.8$ 
& $2.356$ 
& $\textbf{2.250}$
& $2.335$ 
& $2.418$
\\

& $0.7$ 
& $279.1$ 
& $3.233$ 
& $\textbf{3.209}$
& $3.283$ 
& $3.411$
\\
\bottomrule
\end{tabular}
}
\end{table*}

\begin{figure}[ht!]
    \centering
    \begin{minipage}{0.49\textwidth}
        \centering
        \includegraphics[width=\linewidth]{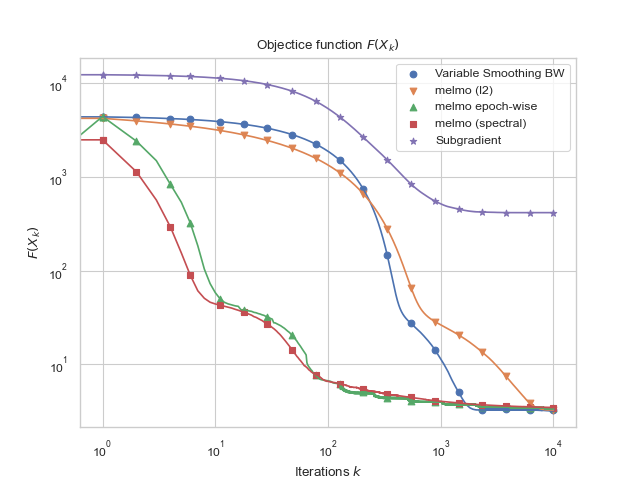}\\
        (Football, $\omega=0.3$)
    \end{minipage}
    \hfill
    \begin{minipage}{0.49\textwidth}
        \centering
        \includegraphics[width=\linewidth]{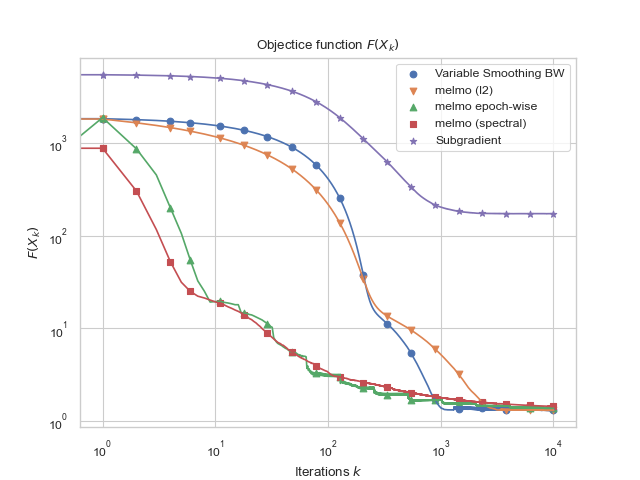}\\
        (Misérables, $\omega=0.3$)
    \end{minipage}
    \caption{Penalized objective $F$ vs.\ iteration (log scale) for masked matrix recovery on Football and Misérables datasets.}
    \label{fig:masked_objectives_sample}
\end{figure}

\begin{figure}[ht!]
    \centering
    \begin{minipage}{0.49\textwidth}
        \centering
        \includegraphics[width=\linewidth]{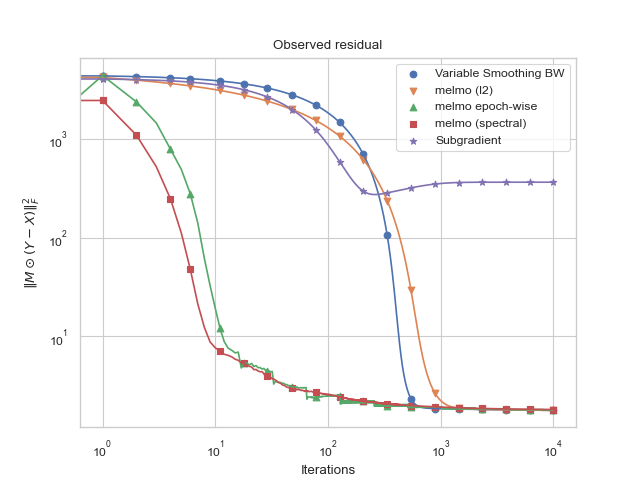}\\
        (Football, $\omega=0.3$)
    \end{minipage}
    \hfill
    \begin{minipage}{0.49\textwidth}
        \centering
        \includegraphics[width=\linewidth]{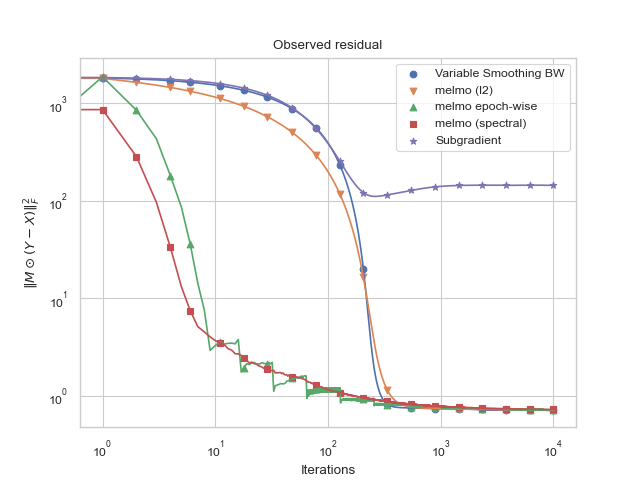}\\
        (Misérables, $\omega=0.3$)
    \end{minipage}

    \caption{Observed residual loss vs.\ iteration (log scale) for masked matrix recovery on Football and Misérables datasets.}
    \label{fig:masked_rLoss_sample}
\end{figure}

Turning to the relative held-out error, Figure~\ref{fig:masked_heldOutLoss_sample} shows that all MELMO variants deliver competitive performance on this metric, with the $\ell_2$-LMO variant closely matching the variable smoothing baseline, which usually achieves the lowest held-out error across all compared methods.
On this held-out error metric, the subgradient method delivers competitive performance, defying its weak behavior on the other evaluation criteria.

\begin{figure}[ht!]
    \centering
    \begin{minipage}{0.49\textwidth}
        \centering
        \includegraphics[width=\linewidth]{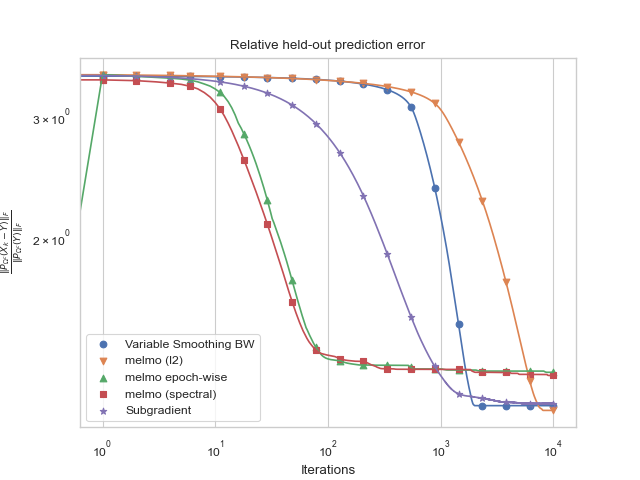}\\
        (Football, $\omega=0.5$)
    \end{minipage}
    \hfill
    \begin{minipage}{0.49\textwidth}
        \centering
        \includegraphics[width=\linewidth]{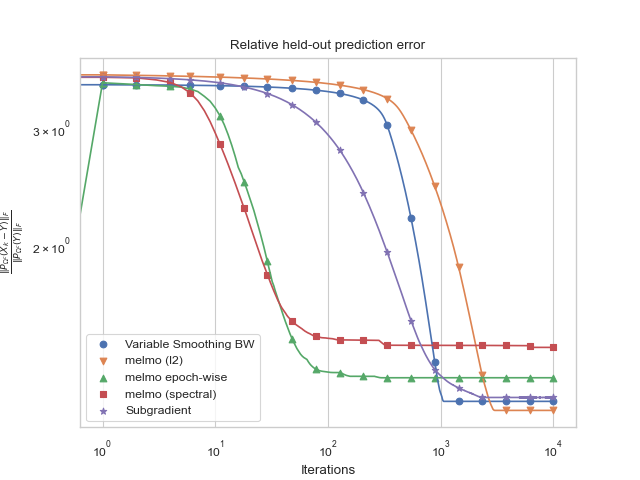}\\
        (Misérables, $\omega=0.5$)
    \end{minipage}
    \caption{Relative held-out error vs.\ iteration (log scale) for masked matrix recovery on Football and Misérables datasets.}
    \label{fig:masked_heldOutLoss_sample}
\end{figure}

Finally, Figures~\ref{fig:masked_proximalGap_sample} and~\ref{fig:masked_dualNorms_sample} illustrate the evolution of the proximal gap and the dual norm of the gradients, respectively. The observed trajectories align with our theoretical analysis and the step size/regularization schedules used by each algorithm.
In particular, caution is required when interpreting the dual norm of the gradients, as these norms are computed under different geometries for different algorithms. Meaningful comparisons are only possible between methods that share the same underlying geometry: for example, the regular MELMO with spectral LMO can be directly compared to its epoch-wise spectral-LMO counterpart. Similarly, the $\ell_2$-LMO variant of MELMO uses a geometry consistent with that of the variable smoothing baseline, making them appropriate peers for comparison.

\begin{figure}[ht!]
    \centering
    \begin{minipage}{0.49\textwidth}
        \centering
        \includegraphics[width=\linewidth]{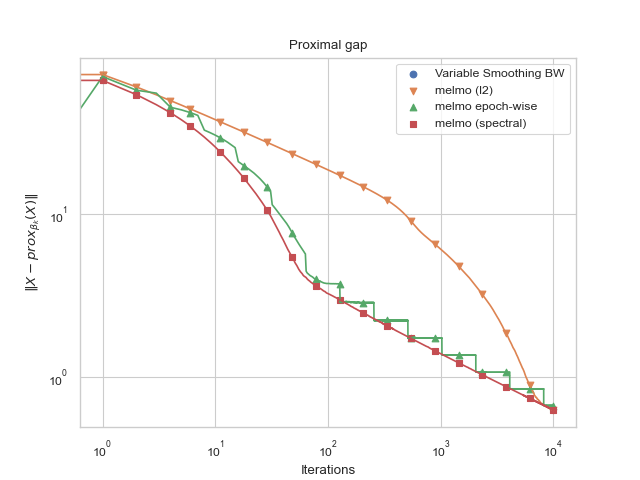}\\
        (Football, $\omega=0.3$)
    \end{minipage}
    \hfill
    \begin{minipage}{0.49\textwidth}
        \centering
        \includegraphics[width=\linewidth]{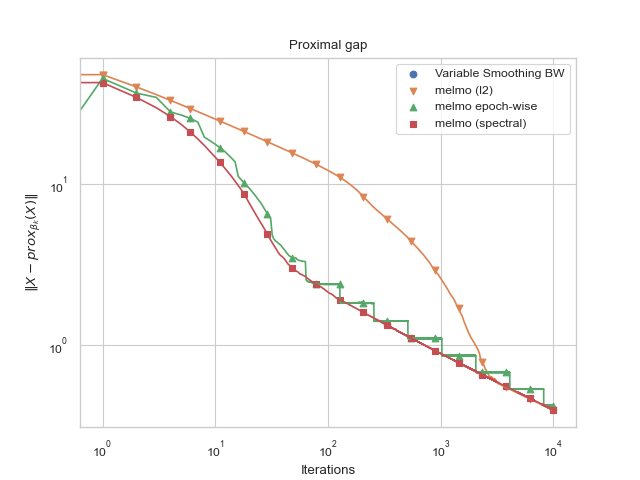}\\
        (Misérables, $\omega=0.3$)
    \end{minipage}
    \caption{Proximal gap vs.\ iteration (log scale) for masked matrix recovery on Football and Misérables datasets.}
    \label{fig:masked_proximalGap_sample}
\end{figure}

\begin{figure}[ht!]
    \centering
    \begin{minipage}{0.49\textwidth}
        \centering
        \includegraphics[width=\linewidth]{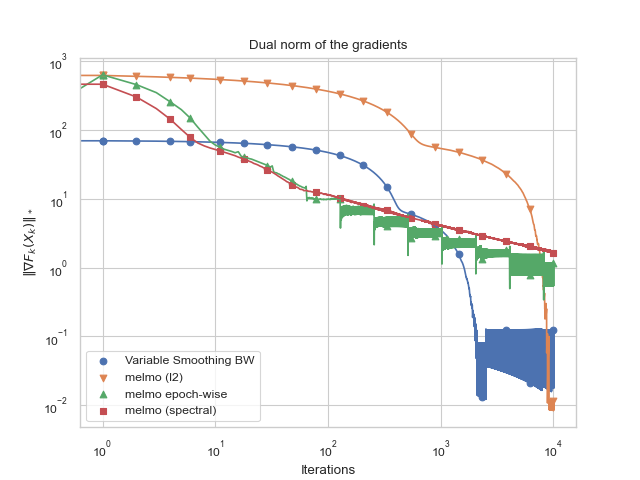}\\
        (Football, $\omega=0.3$)
    \end{minipage}
    \hfill
    \begin{minipage}{0.49\textwidth}
        \centering
        \includegraphics[width=\linewidth]{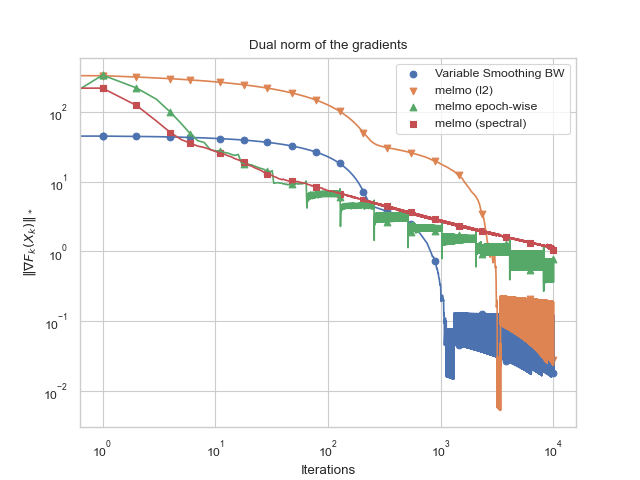}\\
        (Misérables, $\omega=0.3$)
    \end{minipage}
    \caption{Dual norm of the gradients vs.\ iteration (log scale) for masked matrix recovery on Football and Misérables datasets.}
    \label{fig:masked_dualNorms_sample}
\end{figure}

\section{Discussion}
In this work, we presented MELMO, a geometry-aware algorithm for composite optimization problems involving non-smooth terms. By leveraging the Moreau envelope to smooth non-smooth components and incorporating problem-specific geometries through LMOs, our approach unifies and extends several existing optimization techniques.

We established theoretical convergence guarantees for our method under different assumptions on the objective function. For the general case with a Lipschitz-smooth function $f$ and $\rho$-weakly convex function $g$, our analysis yields a family of rates indexed by the schedules $(\gamma_k,\beta_k)$. A balanced schedule gives $\mathcal O(k^{-1/4})$ rates for both the smoothed-gradient norm and the composite stationarity proxy, whereas a more aggressive schedule improves the smoothed-gradient rate to $\mathcal O(k^{-1/3})$ while retaining an $\mathcal O(k^{-1/4})$ bound on the composite proxy. Our analysis also introduced an epoch-wise variant that returns computable $\epsilon$-proxy certificates. With the power schedules inherited from the main theorem, this certification requires $\mathcal O(\epsilon^{-4})$ iterations, while an epoch-constant specialization with per-epoch restarts improves the bound to $\mathcal O(\epsilon^{-3})$. Under surjective $T$, these computable certificates translate into control of the stronger composite stationarity proxy. 

Our algorithm recovers classical steepest descent as a special case while allowing for non-Euclidean geometries that better capture problem structure. Numerical experiments on sparse low-rank matrix factorization show that, under a uniform hyperparameter policy, the balanced schedule $(7/12,1/3)$ is competitive with or improves upon the Euclidean baselines on four of five datasets, while the more aggressive $(2/3,1/4)$ degrade substantially on certain problem instances (its final objective on Olivetti is an order of magnitude worse than the baselines, cf. Table~\ref{tab:performance}). 

The denoising experiment should be read as an empirical extension: the discrete image gradient is not surjective, so the strongest composite-proxy guarantee does not apply, but geometry-aware smoothing nevertheless produces reasonable results in practice. Moreover, unlike the variable smoothing method, our approach does not require knowledge of the gradient Lipschitz constant $L_{\nabla f}$ to set its step size, enhancing its practical applicability.

The masked matrix recovery experiment provides a theory-compliant matrix-valued testbed that validates MELMO's ability to handle non-Euclidean geometries while satisfying all theoretical requirements for convergence. Unlike the sparse low-rank factorization setting, which uses a bilinear loss without a globally Lipschitz gradient, this experiment maintains the Frobenius-norm smoothness guarantee that underpins our convergence analysis. The linear operator $T=I$ is surjective, so the full strength of our composite stationarity certificates applies, and the entrywise proximal operator of the MCP penalty can be computed efficiently via firm thresholding.

Our empirical results on the Football and Misérables datasets confirm that MELMO remains competitive and sometimes better across different observation ratios $\omega\in\{0.3,0.5,0.7\}$. 
The convergence trajectories in Figure~\ref{fig:masked_objectives_sample} show that the spectral LMO variants of MELMO, both regular and epoch-wise, reduce the objective value faster than all baselines in the early iterations, while the $\ell_2$ LMO variant maintains its advantage over the long run to reach the lowest final objective (see Table~\ref{tab:performanceM}).
Note that MELMO algorithms, unlike the variable smoothing baseline, do not require knowledge of the gradient Lipschitz constant $L_{\nabla f}$ to set its step size, enhancing its practical applicability.

This experiment demonstrates that MELMO's geometry-aware smoothing framework works effectively outside of the original low-rank factorization setting, and that it can successfully handle problems that meet all the theoretical requirements of our convergence analysis.

Extending our framework to stochastic settings would broaden its applicability to large-scale machine learning problems.

\newpage

\bibliography{mybib} 

\newpage
\appendix
\onecolumn
\section{Proofs}
\subsection{Preliminaries}

The following is a simple lemma that we include for completeness.
\begin{lemma}
\label{lem:composition_smooth}
Let $g: \mathbb{R}^m \to \mathbb{R} \cup \{+\infty\}$ be a proper, lower semi-continuous, $\rho$-weakly convex function, and let $T: \mathbb{R}^n \to \mathbb{R}^m$ be a linear operator with Euclidean operator norm $\|T\|_{\mathrm{op}}$. Fix a norm $\|\cdot\|$ on $\mathbb{R}^n$ and let $\|\cdot\|_*$ be its dual. Choose constants $c_1,c_2>0$ such that
$$
\|v\|_* \le c_1\|v\|_2
\qquad\text{and}\qquad
\|u\|_2 \le c_2\|u\|
$$
for all $u,v\in\mathbb{R}^n$. Then, for any $\beta>0$ satisfying the conditions of Lemma~\ref{lem:moreau_grad}, $h = g^\beta \circ T$ is $c_1c_2\|T\|_{\mathrm{op}}^2/\beta$-smooth with respect to $\|\cdot\|$:
$$
\|\nabla h(x) - \nabla h(y)\|_* \leq c_1c_2\frac{\|T\|_{\mathrm{op}}^2}{\beta} \|x - y\|, \quad \forall x, y \in \mathbb{R}^n.
$$
\end{lemma}
\begin{proof}
From Lemma~\ref{lem:moreau_grad}, $\nabla g^\beta$ is $1/\beta$-Lipschitz in the Euclidean norm:
\begin{equation}
\|\nabla g^\beta(u) - \nabla g^\beta(v)\|_2 \leq \frac{1}{\beta} \|u - v\|_2, \quad \forall u, v \in \mathbb{R}^m.
\label{eq:g_beta_smooth}
\end{equation}
By the chain rule,
\begin{equation}
\nabla h(x) = T^* \nabla g^\beta(Tx),
\label{eq:h_grad}
\end{equation}
where $T^*$ is the Euclidean adjoint of $T$. For any $x,y\in\mathbb{R}^n$,
\begin{align*}
\|\nabla h(x) - \nabla h(y)\|_* 
&\leq c_1\|\nabla h(x) - \nabla h(y)\|_2\\
&= c_1\|T^* \big( \nabla g^\beta(Tx) - \nabla g^\beta(Ty) \big) \|_2\\
&\leq c_1\|T^*\|_{\mathrm{op}} \|\nabla g^\beta(Tx) - \nabla g^\beta(Ty)\|_2\\
&\leq c_1\|T\|_{\mathrm{op}} \frac{1}{\beta} \|Tx - Ty\|_2\\
&\leq  c_1\frac{\|T\|_{\mathrm{op}}^2}{\beta} \|x - y\|_2\\
&\leq  c_1c_2\frac{\|T\|_{\mathrm{op}}^2}{\beta} \|x - y\|.
\end{align*}

\end{proof}

\begin{lemma}[Bounds on the inexact oracle]\label{lemma:inexact oracle}
Let $\varphi\in\R^n$, choose $d\in \mathrm{lmo}(\varphi)$, and let $\widetilde d$ satisfy
$$
\|\widetilde d-d\|\le \delta<1.
$$

Then the following assertions hold:
\begin{enumerate}
    \item $
\|\widetilde{d}\|^2
\leq (\delta+1)^2
$
    \item $\langle \varphi,\ \widetilde{d}\rangle
\leq
\|\varphi\|_*
(\delta-1)$
    \item $
\langle \varphi,\ \widetilde{d}\rangle 
\leq
(1-\delta)\langle \varphi,\ d\rangle 
$
\end{enumerate}
\end{lemma}
\begin{proof}

First, we have
$$
\|\widetilde{d}\|^2 = \|\widetilde{d}+d-d\|^2 
\leq (\|\widetilde{d}-d\| + \|d\|)^2
\leq (\delta+1)^2
$$
which proves the first claim.

Since $d$ is the LMO, we have $\langle \varphi,\ d\rangle = -\|\varphi\|_*$, so
$$
\langle \varphi,\ \widetilde{d}\rangle 
=
\langle \varphi,\ \widetilde{d} +d-d\rangle 
= 
\langle \varphi,\ d\rangle 
+
\langle \varphi,\ \widetilde{d} -d\rangle 
=
-\|\varphi\|_*
+
\langle \varphi,\ \widetilde{d} -d\rangle 
$$
Moreover, we have
$$
\langle \varphi,\ \widetilde{d} -d\rangle 
\leq
|\langle \varphi,\ \widetilde{d} -d\rangle  |
\leq
\|\varphi\|_*
\|\widetilde{d} -d\|
\leq
\|\varphi\|_*
\delta
$$
Finally, by combining the inequalities, we get
$$
\langle \varphi,\ \widetilde{d}\rangle
\leq
\|\varphi\|_*
(\delta-1)
$$
which proves the second claim.

Alternatively, from the equations above, we have
$$
\langle \varphi,\ \widetilde{d}\rangle 
\leq
\langle \varphi,\ d\rangle 
+
\|\varphi\|_*
\delta
=
\langle \varphi,\ d\rangle 
-
\langle \varphi,\ d\rangle 
\delta
=
(1-\delta)\langle \varphi,\ d\rangle 
$$
which proves the final claim.
\end{proof}
\begin{lemma}\label{lemma:energy diff}
    Let $f\in C^{1,1}(\mathbb{R}^n)$ be Lipschitz-smooth with constant $L_{\nabla f}>0$ with respect to $\|\cdot\|$, $g: \mathbb{R}^m \to \mathbb{R} \cup \{+\infty\}$ a proper, closed, $\rho$-weakly convex, and possibly non-smooth function, and let $T\colon\mathbb{R}^n\to\mathbb{R}^m$ be a linear operator. Assume that $\widetilde{\mathrm{lmo}}$ satisfies Definition~\ref{def:inexact-lmo} with parameter $\delta\in[0,1)$. Then, with $F_k = f + g^{\beta_k}\circ T$ for any $\beta_k>0$ satisfying the conditions of Lemma~\ref{lem:moreau_grad}, we have

    \begin{equation}
        \label{eq:2}
        F_k(x_{k+1}) -  F_k(x_k) 
        \leq 
        \gamma_k \langle \varphi_k,\ \widetilde{d}_k\rangle
        + \frac{L_{F_k} \gamma_k^2}{2} \|\widetilde{d}_k\|^2
        \leq 
        (\delta-1)\gamma_k \|\varphi_k\|_*
        + (\delta+1)^2\frac{L_{F_k} \gamma_k^2}{2}
    \end{equation}
    where $\varphi_k = \nabla F_k(x_k)$, $\widetilde{d}_{k} = \widetilde{\mathrm{lmo}}(\varphi_k)$, and $x_{k+1} = x_k + \gamma_k \widetilde{d}_k$.
\end{lemma}
\begin{proof}
    By assumption on $f$, $g$, and $T$, the function $f$ is $L_{\nabla f}$-smooth with respect to $\|\cdot\|$ and $g^\beta\circ T$ is $c_1c_2\|T\|^2_\text{op}/\beta$-smooth by Lemma~\ref{lem:composition_smooth}, where $c_1,c_2>0$ are the constants from that lemma. Thus, $f + g^\beta\circ T$ is $(L_{\nabla f} + c_1c_2\|T\|^2_\text{op}/\beta)$-smooth with respect to $\|\cdot\|$.

    Let $\varphi_k = \nabla F_k(x_k)$. We have by the descent lemma at the points $x_{k+1}$ and $x_k$, and the smoothness of $ F$
    \begin{equation}
         F_k(x_{k+1}) -  F_k(x_k) \leq \langle \varphi_k,\ x_{k+1} - x_k\rangle
        + \frac{L_{F_k}}{2} \|x_{k+1} - x_k\|^2.
    \end{equation}
    By the definition of $\widetilde{d}_k$ we have
    \begin{equation}
         F_k(x_{k+1}) -  F_k(x_k) \leq 
        \gamma_k \langle \varphi_k,\ \widetilde{d}_k\rangle
        + \frac{L_{F_k} \gamma_k^2}{2} \|\widetilde{d}_k\|^2
    \end{equation}

    Applying Lemma~\ref{lemma:inexact oracle}, we get
    \begin{equation}
         F_k(x_{k+1}) -  F_k(x_k) \leq 
        (\delta - 1)
        \gamma_k \|\varphi_k\|_*
        + (\delta + 1)^2\frac{L_{F_k} \gamma_k^2}{2}
    \end{equation}
\end{proof}

\begin{lemma}\label{lemma:k to k+1}
    Let $f\in C^{1,1}(\mathbb{R}^n)$ be Lipschitz-smooth with constant $L_{\nabla f}>0$, let $g: \mathbb{R}^m \to \mathbb{R} \cup \{+\infty\}$ be a proper, closed, and possibly non-smooth function, and let $T\colon\mathbb{R}^n\to\mathbb{R}^m$ be a linear operator. Let $F_k = f + g^{\beta_k}\circ T$ for any $\beta_k>0$.
    
    Then, if $g$ is $\rho$-weakly convex, we have for any $y\in \R^n$ and $0<\beta_{k+1}\leq\beta_{k}<\rho^{-1}$
    \begin{equation}
        F_{k+1}(y)\leq
        F_k(y)
        + \frac{1}{2} \frac{\beta_{k} - \beta_{k+1}}{\beta_{k+1}} \beta_{k}\|\nabla g^{\beta_{k}}(Ty)\|^2
    \end{equation}

    If $g$ is convex, we have for any $y\in \R^n$ and $\beta_k>0$
    \begin{equation}
        F_{k+1}(y)\leq
        F_k(y)
        +
        \frac{\beta_{k} - \beta_{k+1}}{2} \|\nabla g^{\beta_{k}}(Ty)\|^2
    \end{equation}
\end{lemma}
\begin{proof}
    Since $g$ is $\rho$-weakly convex and $0<\beta_{k+1}\leq\beta_{k}<\rho^{-1}$, Lemma 4.1 in~\cite{bohm2021variable} gives
    $$
        g^{\beta_{k+1}}(y)\leq 
        g^{\beta_{k}}(y)
        + \frac{1}{2} \frac{\beta_{k} - \beta_{k+1}}{\beta_{k+1}} \beta_{k}\|\nabla g^{\beta_{k}}(y)\|^2
    $$
    Applying this inequality at $Ty$ and adding $f(y)$ to both sides gives the claimed bound.

    For the second inequality, Property 7.3 of~\cite{yurtsever2019conditional} states that, for $g$ proper, closed, and convex, we have for any $\beta_{k+1}, \beta_{k} > 0$ and $y\in\R^m$,
    $$
        g^{\beta_{k+1}}(y)\leq 
        g^{\beta_{k}}(y)
        + \frac{\beta_{k} - \beta_{k+1}}{2} \|\nabla g^{\beta_{k}}(y)\|^2
    $$
    Applying this inequality at $Ty$ and adding $f(y)$ to both sides gives the claimed bound.
    
\end{proof}

\begin{lemma}\label{lemma:bounds}
    Let $(\phi_k)_{k\in\mathbb{N}}$ be a sequence of values in $\R_+$. And $C, C_1, C_2, C_3, p, q >0$ some constants such that
    \begin{equation}
        \label{eq:lemma equation}
        \sum_{k = 1}^K
        \frac{C}{k^p}
        \phi_k
        \leq 
        C_1
        +
        \sum_{k = 1}^K
        \frac{C_2}{k^{2p}}
        +
        \frac{C_3}{k^{2p-q}} 
    \end{equation}
    Then, we have the following assertions

    \begin{enumerate}
    
    \item If $p > 1/2$ and $q= 2p-1$, we have

    \begin{equation}
        \min_{j\in[K]}\phi_j
        \leq 
        \mathcal O\!\left(
        \log(K) K^{p-1}
        \right).
    \end{equation}

    \item If $p > 1/2$ and $2p > q > 2p-1$, we have

    \begin{equation}
        \min_{j\in[K]}\phi_j
        \leq
        \mathcal O\!\left(K^{\max\{p-1,\,q-p\}}\right).
    \end{equation}

    \item If $p > 1/2$ and $q < 2p-1$, we have
    \begin{equation}
        \min_{j\in[K]}\phi_j
        \leq
        \mathcal O\!\left(K^{p-1}\right).
    \end{equation}
    
    \item If $p= 1/2$ and $q\in(0, 1/2)$, we have
    \begin{equation}
        \min_{j\in[K]}\phi_j
        \leq 
        \mathcal O\!\left(
        \log(K)K^{-1/2}
        +
        K^{q-1/2}
        \right).
    \end{equation}
    \end{enumerate}
    
\end{lemma}
\begin{proof}
    For the left hand side we have
    \begin{equation}
        \label{eq:kp left side}
        \sum_{k = 1}^K
        \frac{C}{k^p}
        \phi_k
        \geq 
        \sum_{k = 1}^K
        \frac{C}{K^p}
        \min_{j\in[K]}\phi_j
        \geq
        {C}{K^{1-p}}
        \min_{j\in[K]}\phi_j
    \end{equation}

    And we have with the integral test for continuous, positive, decreasing functions, and with $p \neq 1/2$
    \begin{equation*}
        \sum_{k = 1}^K
        k^{-2p}
        \leq 
        1
        +
        \int_1^{K}
        x^{-2p}
        =
        1+\left[\frac{x^{1-2p}}{1-2p}\right]_1^K
        =
        1+\left(\frac{K^{1-2p}-1}{1-2p}\right)
    \end{equation*}

    For $p < 1/2$ we have
    \begin{equation*}
        1+\left(\frac{K^{1-2p}-1}{1-2p}\right)
        \leq
        \frac{K^{1-2p}}{1-2p}
    \end{equation*}
    and for $p > 1/2$ we have 
    \begin{equation*}
        1+\left(\frac{K^{1-2p}-1}{1-2p}\right)
        =
        1+\left(\frac{1-K^{1-2p}}{2p-1}\right)
        \leq
        \frac{2p}{2p-1}
    \end{equation*}
    
    If $p = 1/2$, we have
    \begin{equation}\label{eq:log}
        \sum_{k = 1}^K
        k^{-1}
        \leq 
        1
        +
        \int_1^{K}
        x^{-1}
        =
        1+\left[\log(x)\right]_1^K
        =
        1+\log(K)
    \end{equation}
    
    With the same method and with $0<2p-q$ and $2p-q\neq 1$, we have
    $$
        \sum_{k = 1}^K
        k^{-2p+q}
        \leq 
        1
        +
        \int_1^{K}
        x^{-2p + q}
        =
        1+\left[\frac{x^{1-2p+ q}}{1-2p+ q}\right]_1^K
        =
        \frac{K^{1-2p+ q}-2p+ q}{1-2p+ q}
    $$

    For $q < 2p - 1$ we have
    $$
        =
        \frac{K^{1-2p+ q}-2p+ q}{1-2p+ q}
        \leq
        \frac{2p- q}{2p- q - 1}
    $$

    For $q > 2p - 1$ we have
    $$
        =
        \frac{K^{1-2p+ q}-2p+ q}{1-2p+ q}
        \leq
        \frac{K^{1-2p+ q}}{1-2p+ q}
    $$

    If $2p-q= 1$, we get the Equation~\eqref{eq:log} again.

    Hence, by combining and substituting these bounds in Equation~\eqref{eq:lemma equation}, we get these possible (and permitted) cases

    \begin{enumerate}

    \item If $p > 1/2$ and $q= 2p-1$, we have

    \begin{equation}
        {C}{K^{1-p}}
        \min_{j\in[K]}\phi_j
        \leq 
        C_1
        +
        C_{3}
        +
        C_{2}
        \frac{2p}{2p-1}
        +
        C_3
        \log(K)
    \end{equation}
    Which gives
    \begin{equation}
        \min_{j\in[K]}\phi_j
        \leq 
        \left(
        \frac{C_1 + C_3}{C}
        +
        \frac{C_2 2p}{C(2p-1)}
        \right)
        K^{p-1}
        +
        \frac{C_3}{C}
        \log(K)
        K^{p-1}
    \end{equation}

    \item If $p > 1/2$ and $2p > q > 2p-1$, we have

    \begin{equation}
        {C}{K^{1-p}}
        \min_{j\in[K]}\phi_j
        \leq
        C_1
        +
        C_{2}
        \frac{2p}{2p-1}
        +
        \frac{C_3}{1-2p+q}K^{1-2p+q}
    \end{equation}
    Which gives
    \begin{equation}
        \min_{j\in[K]}\phi_j
        \leq
        \left(
        C_1
        +
        C_{2}
        \frac{2p}{2p-1}
        \right)
        \frac{K^{p-1}}{C}
        +
        \frac{C_3}{C(1-2p+q)}
        K^{q-p}
    \end{equation}

    \item If $p > 1/2$ and $q < 2p-1$, we have

    \begin{equation}
        {C}{K^{1-p}}
        \min_{j\in[K]}\phi_j
        \leq
        C_1
        +
        C_{2}
        \frac{2p}{2p-1}
        +
        C_3
        \frac{2p- q}{2p- q - 1}
    \end{equation}
    Which gives
    \begin{equation}
        \min_{j\in[K]}\phi_j
        \leq
        \left(
        C_1
        +
        C_{2}
        \frac{2p}{2p-1}
        +
        C_3
        \frac{2p- q}{2p- q - 1}
        \right)
        \frac{K^{p-1}}{C}
    \end{equation}
    
    \item If $p= 1/2$ and $q\in(0, 1/2)$, we have

    \begin{equation}
        \label{eq:Okp bounds p log}
        {C}{K^{1/2}}
        \min_{j\in[K]}\phi_j
        \leq 
        C_1
        +
        C_{2}
        +
        C_{2}
        \log(K)
        +
        C_3
        \frac{K^{q}}{q}
    \end{equation}
    Which gives
    \begin{equation}
        \min_{j\in[K]}\phi_j
        \leq 
        \frac{C_1 + C_2}{C}
        K^{-1/2}
        +
        \frac{C_2}{C}
        \log(K)
        K^{-1/2}
        +
        \frac{C_3}{Cq}
        K^{q-1/2}
    \end{equation}
    \end{enumerate}
    
\end{proof}

\subsection{Proof of Theorem~\ref{th:1}}\label{proof:1}
\begin{proof}
Write $\varphi_k=\nabla F_k(x_k)$. By Lemmas~\ref{lemma:energy diff} and~\ref{lemma:k to k+1}, and because $g$ is $L_g$-Lipschitz,
\begin{equation}
    \label{eq:3}
    F_{k+1}(x_{k+1}) -  F_k(x_k) \leq
    \frac{1}{2}(\beta_{k} - \beta_{k+1}) \frac{\beta_{k}}{\beta_{k+1}} L_g^2
    -
    (1-\delta)
    \gamma_k \|\varphi_k\|_*
    + (1+\delta)^2\frac{L_{F_k} \gamma_k^2}{2}.
\end{equation}
For $\beta_k=\beta k^{-q}$, the ratio $\beta_k/\beta_{k+1}=((k+1)/k)^q$ is bounded by $2^q$. Summing~\eqref{eq:3} from $k=1$ to $K$ and using $\sum_{k=1}^K(\beta_k-\beta_{k+1})\le \beta_1$ gives
\begin{equation}
    \label{eq:rearranging}
    (1-\delta)\sum_{k = 1}^K
    \gamma_k \|\varphi_k\|_*
    \leq
    2^{q-1}\beta_1 L_g^2
    +
    F_1(x_1)
    -
    F_{K+1}(x_{K+1})
    +
    (1+\delta)^2
    \sum_{k = 1}^K
    \frac{L_{F_k} \gamma_k^2}{2}.
\end{equation}
Since the Moreau envelope is monotone in the smoothing parameter and $\beta_{k+1}\le\beta_k$, the functions $F_k$ are pointwise nondecreasing. Hence
$$
F_{K+1}(x_{K+1})\ge F_{K+1}^{\star}\ge F_1^\star.
$$
Using
$$
L_{F_k}=L_{\nabla f}+c_1c_2\frac{\|T\|_{\mathrm{op}}^2}{\beta_k},
\qquad
\gamma_k=\gamma k^{-p},
\qquad
\beta_k=\beta k^{-q},
$$
with the constants $c_1,c_2$ from Lemma~\ref{lem:composition_smooth}, we obtain
\begin{equation}
    \label{eq:Okp gamma}
    (1-\delta)
    \sum_{k = 1}^K
    \frac{\gamma}{k^p}
    \|\varphi_k\|_*
    \leq
    C_0
    +
    (1+\delta)^2
    \sum_{k = 1}^K
    \left(
    \frac{L_{\nabla f}\gamma^2}{2k^{2p}}
    +
    \frac{\gamma^2c_1c_2\|T\|_{\mathrm{op}}^2}{2\beta k^{2p-q}}
    \right),
\end{equation}
where $C_0=2^{q-1}\beta L_g^2+F_1(x_1)-F_1^\star$ is independent of $K$.

The assumptions $q<p$ and $p>0$ imply $q<2p$, so Lemma~\ref{lemma:bounds} applies to~\eqref{eq:Okp gamma}. In the non-boundary case $p\neq1/2$ and $2p-q\neq1$, it gives
$$
\min_{j\in[K]}\|\nabla F_j(x_j)\|_*
=
\mathcal O\!\left(K^{\max\{p-1,\,-p,\,q-p\}}\right).
$$

Now assume that $T$ is surjective and let $z_j$ and $\tilde x_j$ be as in the statement of the theorem. Two points require care.

First, the translation to composite stationarity. Since $T\tilde x_j = z_j$, $\|\tilde x_j - x_j\|_2\le\sigma_{\min}(T)^{-1}\|Tx_j-z_j\|_2$, and $\nabla g^{\beta_j}(Tx_j)\in\partial g(z_j)$, we can estimate
\begin{align*}
\mathrm{dist}_2\big({-\nabla f(\tilde x_j)},\, T^*\partial g(z_j)\big)
&\le \big\|\nabla f(\tilde x_j) + T^*\nabla g^{\beta_j}(Tx_j)\big\|_2\\
&\le \|\nabla F_j(x_j)\|_2 + \|\nabla f(\tilde x_j)-\nabla f(x_j)\|_2\\
&\le c_1\|\nabla F_j(x_j)\|_* + c_2c_3\, L_{\nabla f}\,\|\tilde x_j - x_j\|_2\\
&\le c_1\|\nabla F_j(x_j)\|_* + c_2c_3\, L_{\nabla f}\,\sigma_{\min}(T)^{-1}\beta L_g\, j^{-q},
\end{align*}
where the last step uses Theorem~\ref{thm:gap original prox}, and $c_1,c_2, c_3$ are the norm-equivalence constants so that $\|v\|_2\le c_1\|v\|_*$, $\|v\|_2\le c_2\|v\|$ and $\|u\|\le c_3\|u\|_2$.

Second, the two terms above must be small at a \emph{common} index $j$: the index minimizing $\|\nabla F_j(x_j)\|_*$ over $[K]$ could occur early, where $\beta_j$ is still large. To handle this, restrict the minimum to the last half of the horizon: since every coefficient satisfies $\gamma k^{-p}\ge \gamma K^{-p}$ and the window $\{\lceil K/2\rceil,\dotsc,K\}$ contains at least $K/2$ indices,
$$  
\sum_{k=1}^K \frac{\gamma}{k^p}\|\varphi_k\|_*
\;\ge\; \sum_{k=\lceil K/2\rceil}^{K}\frac{\gamma}{k^p}\|\varphi_k\|_*
\;\ge\; \frac{\gamma}{2}\,K^{1-p}\min_{\lceil K/2\rceil\le j\le K}\|\varphi_j\|_* ,
$$
so~\eqref{eq:Okp gamma} also gives $\min_{\lceil K/2\rceil\le j\le K}\|\varphi_j\|_* = \mathcal O\!\big(K^{\max\{p-1,-p,q-p\}}\big)$, at the cost of a factor $2$. For every $j\ge\lceil K/2\rceil$ we have $j^{-q}\le 2^q K^{-q}$, so evaluating the display above at the minimizing index of the last half yields
$$
\min_{j\in[K]}\mathrm{dist}_2\big({-\nabla f(\tilde x_j)},\,T^*\partial g(z_j)\big)
\le
\mathcal O\!\left(K^{\max\{p-1,\,-p,\,q-p,\,-q\}}\right),
$$
with $\|\tilde x_j - x_j\|_2 = \mathcal O(K^{-q})$ at that same index.

\end{proof}
    
\subsection{Proof of Corollary~\ref{co:1}}\label{proof:co1}
\begin{proof}
    When $g$ is convex, Lemma~\ref{lemma:k to k+1} gives
    $$
    F_{k+1}(x_{k+1})-F_k(x_k)
    \le
    \frac{\beta_k-\beta_{k+1}}{2}L_g^2
    -
    (1-\delta)\gamma_k\|\varphi_k\|_*
    +
    (1+\delta)^2\frac{L_{F_k}\gamma_k^2}{2}.
    $$
    This is the same recurrence used in the proof of Theorem~\ref{th:1}, with a smaller smoothing-change term and without any weak-convexity restriction on $\beta$. Repeating that proof verbatim gives the stated bounds and hence the two schedule corollaries.
\end{proof}

\subsection{Proof of theorem~\ref{th:1 fixed K}}\label{proof:1 fixed K}
\begin{proof}
    By Lemmas~\ref{lemma:energy diff} and~\ref{lemma:k to k+1}, and the assumption that $g$ is $L_g$-Lipschitz, we have
    \begin{equation}
         F_{k+1}(x_{k+1}) -  F_k(x_k) \leq 
        \frac{1}{2}(\beta_{k} - \beta_{k+1}) \frac{\beta_{k}}{\beta_{k+1}} L_g^2
        +
        (\delta - 1)
        \gamma_k \|\varphi_k\|_*
        + (\delta + 1)^2\frac{L_{F_k} \gamma_k^2}{2}
    \end{equation}
    
    Moreover, assuming $\frac{\beta_{k}}{\beta_{k+1}}\leq 2$, we obtain
    \begin{equation}
         F_{k+1}(x_{k+1}) -  F_k(x_k) \leq 
        (\beta_{k} - \beta_{k+1}) L_g^2
        +
        (\delta - 1)
        \gamma_k \|\varphi_k\|_*
        + (\delta + 1)^2\frac{L_{F_k} \gamma_k^2}{2}
    \end{equation}
    Summing over $k=1$ to $K$ and assuming $F_1^\star > -\infty$, the telescoping terms give
    \begin{align}
         F_{K+1}(x_{K+1}) - F_1(x_1) 
         &\leq 
        (\beta_1 - \beta_{K+1}) L_g^2
        +
        \sum_{k = 1}^K
        (\delta + 1)^2\frac{L_{F_k} \gamma_k^2}{2}
        +
        (\delta - 1)\gamma_k \|\varphi_k\|_*
        \\
        &\leq 
        \beta_{1} L_g^2
        +
        \sum_{k = 1}^K
        (\delta + 1)^2\frac{L_{F_k} \gamma_k^2}{2}
        +
        (\delta - 1)\gamma_k \|\varphi_k\|_*
    \end{align}

    By rearranging the terms, we get
    \begin{equation}
        (1-\delta)\sum_{k = 1}^K
        \gamma_k \|\varphi_k\|_*
        \leq 
        \beta_{1} L_g^2
        +
         F_1(x_1)
        -
         F_{K+1}(x_{K+1})
        +
        (\delta + 1)^2
        \sum_{k = 1}^K
        \frac{L_{F_k} \gamma_k^2}{2}
    \end{equation}

    With $L_{F_k} =  L_{\nabla f} + c_1c_2\|T\|^2_\text{op}/\beta_k$, where $c_1,c_2>0$ are the constants from Lemma~\ref{lem:composition_smooth}, take $\gamma_k= \gamma/K^{2/3}$ and $\beta_k = \beta/K^{1/3}$.
    Furthermore, by the monotonicity of $F_k$,
    $$
    F_{K+1}(x_{K+1}) \geq \min_x F_{K+1}(x) =: F_{K+1}^{\star} \geq F_{1}^{\star}
    $$
    we have
    \begin{equation}
        (1-\delta)
        \sum_{k = 1}^K
        \frac{\gamma}{K^{2/3}}
        \|\varphi_k\|_*
        \leq 
        \underbrace{\beta_{1} L_g^2
        +
         F_1(x_1)
        -
         F_{1}^{\star}}_{=: C}
        +
        (\delta + 1)^2
        \sum_{k = 1}^K
        \frac{L_{\nabla f} \gamma^2}{2K^{4/3}}
        +
        \frac{\gamma^2c_1c_2\|T\|^2_\text{op}}{2K\beta} 
    \end{equation}

    Hence, we obtain
    \begin{equation}
        (1-\delta)
        {\gamma}
        K^{1/3}
        \min_{k\in[K]}
        \|\varphi_k\|_*
        \leq 
        C
        +
        (\delta + 1)^2
        K^{-1/3}
        \frac{L_{\nabla f} \gamma^2}{2}
        +
        \frac{\gamma^2c_1c_2\|T\|^2_\text{op}}{2\beta} 
    \end{equation}
    By rearranging the terms, we get 
    $$
    \min_{k\in[K]}\|\varphi_k\|_*
    \leq
    \mathcal O(K^{-1/3})
    $$
    
    Here $C$ is bounded independently of $K$: since the Moreau envelope minorizes $g$ and is monotone in the smoothing parameter, and $\beta_1 = \beta K^{-1/3}\le\beta$, we have $F_1(x_1)\le F(x_1)$ and $F_1^\star \ge \inf_x \{f + g^{\beta}\circ T\}(x)$, which is finite by assumption.

    If $T$ is surjective, the stationarity translation established in the proof of Theorem~\ref{th:1} gives, with $z_k = \mathrm{prox}_{\beta_k g}(Tx_k)$ and $\tilde x_k = x_k + T^*(TT^*)^{-1}(z_k - Tx_k)$,
    $$
    \mathrm{dist}_2(-\nabla f(\tilde x_k), T^*\partial g(z_k))
    \le
    c_1\|\varphi_k\|_* + c_2c_3\,L_{\nabla f}\,\sigma_{\min}(T)^{-1}\beta L_g\,K^{-1/3}
    \qquad\text{for every } k\in[K],
    $$
    since the smoothing parameter is constant, so, in contrast with Theorem~\ref{th:1}, no index-alignment argument is needed. Taking the minimizing index of $\|\varphi_k\|_*$ yields
    $$
    \min_{k\in[K]}\mathrm{dist}_2(-\nabla f(\tilde x_k), T^*\partial g(z_k))
    \leq
    \mathcal O(K^{-1/3})
    $$
\end{proof}

\subsection{Proof of Theorem~\ref{th:Epoch-wise} and Proposition~\ref{prop:epoch-wise-fast}}\label{proof:Epoch-wise}

We first record two observations used by both proofs.

First, whenever Algorithm~\ref{alg:epoch_vs} returns a point, that point is an $\epsilon$-proxy certificate \emph{by construction}, because both quantities in the stopping rule are evaluated explicitly before returning. The content of Theorem~\ref{th:Epoch-wise} and of Proposition~\ref{prop:epoch-wise-fast} is therefore an upper bound on the number of iterations before the stopping rule must trigger.

Second, during epoch $l$, Algorithm~\ref{alg:epoch_vs} evaluates the gradient norms $\|\nabla F_{k+1}(x_{k+1})\|_*$ for $k = 2^l,\dotsc,2^{l+1}-1$, i.e., at the indices
$$
W_l := \{2^l+1,\, \dotsc,\, 2^{l+1}\},
$$
and the value of $S_l$ at the end of epoch $l$ equals $\min_{j\in W_l}\|\varphi_j\|_*$, where, as before, $\varphi_j = \nabla F_j(x_j)$. 

\begin{proof}[Proof of Theorem~\ref{th:Epoch-wise}]
Throughout, $(p,q) = (2/3,1/4)$, i.e., $\gamma_k = \gamma k^{-2/3}$ and $\beta_k = \beta k^{-1/4}$, and $L_{F_k} = L_{\nabla f} + c_1c_2\|T\|^2_{\mathrm{op}}/\beta_k$ with $c_1,c_2$ the constants from Lemma~\ref{lem:composition_smooth}. By Lemmas~\ref{lemma:energy diff} and~\ref{lemma:k to k+1}, the $L_g$-Lipschitz continuity of $g$, and $\beta_k/\beta_{k+1} = ((k+1)/k)^{1/4}\le 2$, we have for every $k\ge1$
\begin{equation}\label{eq:one step epoch}
     F_{k+1}(x_{k+1}) -  F_k(x_k) \leq
    (\beta_{k} - \beta_{k+1}) L_g^2
    -
    (1-\delta)
    \gamma_k \|\varphi_k\|_*
    + (\delta + 1)^2\frac{L_{F_k} \gamma_k^2}{2}.
\end{equation}

\emph{Step 1: uniform bound on the objective drift.}
Dropping the negative term in~\eqref{eq:one step epoch} and summing from $k=1$ to $K-1$ gives, for every $K\ge 1$, using $\sum_{k\ge1}(\beta_k-\beta_{k+1})\le\beta$,
\begin{align*}
    F_K(x_K) &\le
F_1(x_1) + \beta L_g^2
+ \frac{(\delta+1)^2}{2}\left(L_{\nabla f}\gamma^2\sum_{k\ge1}k^{-4/3} + \frac{c_1c_2\|T\|^2_{\mathrm{op}}\gamma^2}{\beta}\sum_{k\ge1}k^{-13/12}\right)\\
&=: F_1(x_1) + D\\
\end{align*}

where $D<\infty$ because both series converge: $2p = 4/3>1$ and $2p-q = 13/12>1$.
Since the schedules are nonincreasing, the functions $F_k$ are pointwise nondecreasing in $k$, so $F_K(x_K) \ge F_K^\star \ge F_1^\star$ for every $K$. Consequently, for every epoch $l$,

\begin{equation}\label{eq:uniform drift}
F_{2^l+1}(x_{2^l+1}) - F_{2^{l+1}+1}(x_{2^{l+1}+1})
\;\le\; F_1(x_1) + D - F_1^\star \;=:\; \Delta \;<\;\infty .
\end{equation}

\emph{Step 2: per-epoch bound on the smoothed gradient.}
Summing~\eqref{eq:one step epoch} over $k\in W_l$ (that is, from $k = 2^l+1$ to $2^{l+1}$, matching the indices evaluated by the algorithm), telescoping the $F$-terms, and using~\eqref{eq:uniform drift} together with $\sum_{k\in W_l}(\beta_k-\beta_{k+1})\le\beta_{2^l}$ yields

\begin{equation}
\label{eq:sum epoch}
(1-\delta)
\sum_{k \in W_l}
\frac{\gamma}{k^{2/3}}
\|\varphi_k\|_*
\leq
\beta_{2^l} L_g^2 + \Delta
+
(\delta + 1)^2
\sum_{k \in W_l}
\left(
\frac{L_{\nabla f} \gamma^2}{2k^{4/3}}
+
\frac{\gamma^2c_1c_2\|T\|^2_{\mathrm{op}}}{2\beta\, k^{13/12}}
\right).
\end{equation}

For the left-hand side, since $k \le 2^{l+1}$ for every $k\in W_l$ and $|W_l| = 2^l$,

\begin{equation}
    \label{eq:kp left side epoch}
    (1-\delta)
    \sum_{k \in W_l}
    \frac{\gamma}{k^{2/3}}
    \|\varphi_k\|_*
    \geq
    (1-\delta)\,
    \frac{\gamma}{2^{2/3}}\,(2^{l})^{1/3}
    \min_{j\in W_l}\|\varphi_j\|_* .
\end{equation}
For the right-hand side, comparing each summand with the integral of the decreasing function it samples ($k^{-a}\le\int_{k-1}^{k} x^{-a}\,dx$),
$$
\sum_{k \in W_l} k^{-4/3}
\le \int_{2^l}^{\infty} x^{-4/3}\,dx
= 3\,(2^l)^{-1/3},
\qquad
\sum_{k \in W_l} k^{-13/12}
\le \int_{2^l}^{\infty} x^{-13/12}\,dx
= 12\,(2^l)^{-1/12}.
$$
Combining the three displays and using $\beta_{2^l} = \beta(2^l)^{-1/4}$, we arrive at
\begin{equation}
\begin{aligned}
    \label{eq:sum epoch 2}
    \min_{j \in W_l}
    \|\varphi_j\|_*
    \leq\ &
    \frac{2^{2/3}\,\Delta}{\gamma(1-\delta)}
    \, (2^{l})^{-1/3}
    +
    \frac{2^{2/3}\,\beta L_g^2}{\gamma(1-\delta)}\, (2^l)^{-7/12}
    +
    (\delta + 1)^2
    \frac{3\cdot 2^{2/3}\, L_{\nabla f}\, \gamma}{2(1-\delta)}\,
    (2^{l})^{-2/3}
    \\&+
    (\delta + 1)^2
    \frac{6\cdot 2^{2/3}\,\gamma\, c_1c_2\|T\|^2_{\mathrm{op}}}{\beta(1-\delta)}\,
    (2^{l})^{-5/12}
    \;\le\; A\,(2^l)^{-1/3}
\end{aligned}
\end{equation}
for a constant $A$ depending only on $\Delta,\gamma,\beta,\delta,L_g,L_{\nabla f}$, and $c_1c_2\|T\|^2_{\mathrm{op}}$ (the exponent $-1/3$ is the largest of the four).

\emph{Step 3: proximal gap and iteration count.}
By Theorem~\ref{thm:gap original prox}, for every $j \in W_l$,

\begin{equation}
\label{eq:prox bound epoch}
\|Tx_j - z_j\|_2 \leq \beta L_g\, j^{-1/4}
\leq
\beta L_g\, (2^l)^{-1/4},
\qquad z_j = \prox{\beta_j g}{Tx_j}.
\end{equation}

Let $l^*$ be the first epoch such that both
$$
A\,(2^{l^*})^{-1/3}\le\epsilon
\qquad\text{and}\qquad
\beta L_g\, (2^{l^*})^{-1/4}\le\epsilon,
\qquad\text{so that}\qquad
2^{l^*} \le 2\max\big\{(A/\epsilon)^{3},\, (\beta L_g/\epsilon)^{4}\big\}.
$$

Suppose the algorithm has not already stopped before epoch $l^*$. Within epoch $l^*$, consider the first index $j^*\in W_{l^*}$ at which the running minimum $S_{l^*}$ reaches its final value $\min_{j\in W_{l^*}}\|\varphi_j\|_*\le\epsilon$. At that iteration the algorithm sets $S_{l^*}\le\epsilon$ and evaluates the proximal gap at $j^*$, which satisfies~\eqref{eq:prox bound epoch} and hence is at most $\epsilon$; the stopping rule therefore triggers and the algorithm returns $x_{j^*}=x_{j_{l^*}}$. The total number of iterations performed up to that point is at most $2^{l^*+1} = \mathcal O(\epsilon^{-4})$, which proves Theorem~\ref{th:Epoch-wise}.
\end{proof}

\begin{proof}[Proof of Proposition~\ref{prop:epoch-wise-fast}]
Consider the epoch-constant specialization with restarts: at the start of epoch $l$ the iterate is reset to a point $x^0_l$ with $F(x^0_l)\le F(x_1)$, and throughout the epoch one uses the constant parameters
\begin{align*}
&\beta_{(l)} := \beta 2^{-(l+1)/3},
\qquad
&\gamma_{(l)} := \gamma 2^{-2(l+1)/3},
\\
&F_{(l)} := f + g^{\beta_{(l)}}\circ T,
\qquad
&L_{(l)} := L_{\nabla f} + \frac{c_1c_2\|T\|^2_{\mathrm{op}}}{\beta_{(l)}} .
\end{align*}
Because the smoothing parameter is constant within the epoch and each epoch starts afresh, every epoch is a self-contained fixed-horizon run of length $2^l$ in the sense of Theorem~\ref{th:1 fixed K}: there is no smoothing-change term, and summing the descent inequality of Lemma~\ref{lemma:energy diff} over the $2^l$ iterations of epoch $l$ gives
\begin{equation}
\label{eq:sum epoch alt}
(1-\delta)\,\gamma_{(l)}
\sum_{k\in E_l}
\|\varphi_k\|_*
\leq
F_{(l)}(x^0_l) - F_{(l)}(x^{\mathrm{end}}_l)
+
(\delta + 1)^2\, 2^l\,
\frac{L_{(l)}\gamma_{(l)}^2}{2},
\end{equation}
where $E_l$ collects the $2^l$ points at which steps are taken during epoch $l$ (starting from $x^0_l$), $x^{\mathrm{end}}_l$ is the final point of the epoch, and $\varphi_k = \nabla F_{(l)}(x_k)$.

Two facts make the right-hand side of~\eqref{eq:sum epoch alt} bounded uniformly in $l$.
First, since the Moreau envelope minorizes $g$, we have $F_{(l)}(x^0_l)\le F(x^0_l)\le F(x_1)$; and since $\beta_{(l)}\le\beta$ implies $F_{(l)}\ge F_1$ pointwise (monotonicity of the envelope in the smoothing parameter), we also have $F_{(l)}(x^{\mathrm{end}}_l)\ge F_1^\star$. Hence
$$
F_{(l)}(x^0_l) - F_{(l)}(x^{\mathrm{end}}_l) \;\le\; F(x_1)-F_1^\star \;=:\; \Delta' ,
$$
independently of $l$. 
And, we have
$$
2^l\,\frac{L_{(l)}\gamma_{(l)}^2}{2}
=
\frac{\gamma^2 L_{\nabla f}}{2}\,2^{\,l-\frac{4(l+1)}{3}}
+
\frac{\gamma^2 c_1c_2\|T\|^2_{\mathrm{op}}}{2\beta}\,2^{\,l-(l+1)}
\le
\frac{\gamma^2 L_{\nabla f}}{2}
+
\frac{\gamma^2 c_1c_2\|T\|^2_{\mathrm{op}}}{4\beta}
=: Q,
$$
again independently of $l$.

For the left-hand side of~\eqref{eq:sum epoch alt}, note that $2^l\,\gamma_{(l)} = \gamma\,2^{(l-2)/3} = \gamma\,2^{-2/3}(2^l)^{1/3}$. The gradient norms evaluated by the algorithm during epoch $l$ are those of the iterates produced within the epoch; all of them except the last belong to $E_l$. Therefore, discarding at most one index and using $|E_l|-1 = 2^l-1\ge 2^{l-1}$ for $l\ge1$,
$$
(1-\delta)\,\gamma_{(l)}\sum_{k\in E_l}\|\varphi_k\|_*
\;\ge\;
(1-\delta)\,\frac{\gamma}{2\cdot 2^{2/3}}\,(2^l)^{1/3}\,
\min_{j}\|\varphi_j\|_* ,
$$
where the minimum runs over the evaluated iterates of epoch $l$ except the last one; since the algorithm's $S_l$ is a minimum over a superset of these indices, $S_l$ obeys the same bound. Combining the three previous displays,
\begin{equation}
S_l \;\le\; \min_{j}\|\varphi_j\|_*
\;\le\;
\frac{2^{5/3}\big(\Delta' + (\delta+1)^2 Q\big)}{\gamma(1-\delta)}\,(2^l)^{-1/3}
\;=:\; A'\,(2^l)^{-1/3}.
\end{equation}
Moreover, with the epoch-constant smoothing parameter, the argument of Theorem~\ref{thm:gap original prox} gives, for every iterate $x_j$ of epoch $l$,
\begin{equation}
\|Tx_j - z_j\|_2 \leq \beta_{(l)} L_g = \beta L_g\, 2^{-(l+1)/3} \le \beta L_g\, (2^l)^{-1/3},
\qquad z_j = \prox{\beta_{(l)} g}{Tx_j}.
\end{equation}
Exactly as in Step 3 of the previous proof, the stopping rule must trigger no later than the first epoch $l^*$ with $\max\{A',\,\beta L_g\}\,(2^{l^*})^{-1/3}\le\epsilon$, i.e., $2^{l^*} = \mathcal O(\epsilon^{-3})$, and the total number of iterations up to and including that epoch is at most $\sum_{l\le l^*}2^l\le 2^{l^*+1} = \mathcal O(\epsilon^{-3})$. This proves Proposition~\ref{prop:epoch-wise-fast}.
\end{proof}

\section{Additional Figures and Tables}

\begin{figure}
    \centering
    \caption{Impact of the Newton--Schulz iterations on the performance of the MELMO algorithm. Here, we compare the objective function curve of the by default 6-iterations Newton--Schulz method with the one given by the 2-iterations method.}
    \label{fig:impact_of_iterations}
    \centering
    \begin{minipage}{0.49\textwidth}
        \centering
        \includegraphics[width=\linewidth]{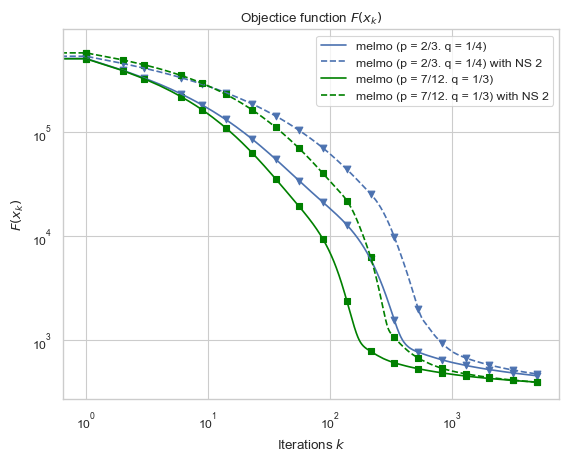}
        \\
        (r = 10)
    \end{minipage}
    \hfill
    \begin{minipage}{0.49\textwidth}
        \centering
        \includegraphics[width=\linewidth]{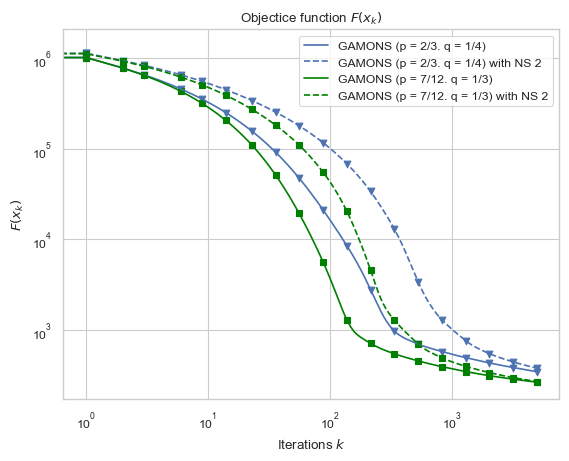}
        \\
        (r = 20)
    \end{minipage}
    \hfill
    \begin{minipage}{0.49\textwidth}
        \centering
        \includegraphics[width=\linewidth]{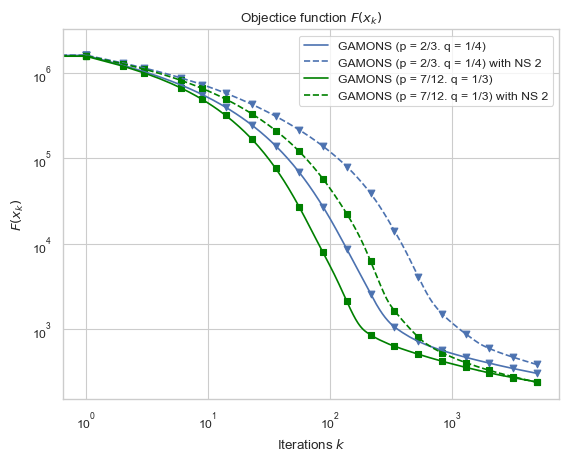}
        \\
        (r = 30)
    \end{minipage}
\end{figure}

\begin{table}[ht]
\centering
\caption{Hyperparameter settings for the sparse low-rank experiments. All values are fixed across datasets; no per-instance tuning is performed.}
\label{tab:hyperparameters_separate}
\begin{tabular}{lll}
\toprule
\textbf{Algorithm} & \textbf{Hyperparameter} & \textbf{Value} \\
\midrule
\multirow{1}{*}{Variable Smoothing}
    & $\beta_0$    & $1$ \\
\midrule
\multirow{2}{*}{MELMO}
    & $\gamma$                & $1$ \\
    & $\beta_0$    & $1$ \\
\midrule
\multirow{1}{*}{Sub-gradient}
    & Step size         & $5\times 10^{-4}$ \\
\midrule
\multirow{2}{*}{Shared}
    & Max iterations               & $5 000$ \\
    & MCP $\mu$               & $3$ \\
    & MCP $\lambda$            & $1$ \\
\bottomrule
\end{tabular}
\end{table}

\begin{table}[ht]
\centering
\caption{Hyperparameter settings for the image denoising experiments. All values are fixed across datasets; no per-instance tuning is performed.}
\label{tab:hyperparameters_denoising}
\begin{tabular}{lll}
\toprule
\textbf{Algorithm} & \textbf{Hyperparameter} & \textbf{Value} \\
\midrule
\multirow{1}{*}{Variable Smoothing}
    & $\beta_0$    & $1$ \\
\midrule
\multirow{2}{*}{MELMO}
    & $\gamma$                & $1$ \\
    & $\beta_0$    & $1$ \\
\midrule
\multirow{2}{*}{Shared}
    & Max iterations               & $2^{13} = 8192$ \\
    & MCP $\mu$               & $2$ \\
    & MCP $\lambda$            & $0.05$ \\
\bottomrule
\end{tabular}
\end{table}

\begin{table}[ht]
\centering
\caption{Hyperparameter settings for the image masked matrix recovery experiments. All values are fixed across datasets; no per-instance tuning is performed.}
\label{tab:hyperparameters_mask}
\begin{tabular}{lll}
\toprule
\textbf{Algorithm} & \textbf{Hyperparameter} & \textbf{Value} \\
\midrule
\multirow{1}{*}{Variable Smoothing}
    & $\beta_0$    & $1$ \\
\midrule
\multirow{2}{*}{MELMO}
    & $\gamma$                & $1$ \\
    & $\beta_0$    & $1$ \\
\midrule
\multirow{1}{*}{Sub-gradient}
    & Step size         & $5\times 10^{-3}$ \\
\midrule
\multirow{2}{*}{Shared}
    & Max iterations               & $10 000$ \\
    & MCP $\mu$               & $3$ \\
    & MCP $\lambda$            & $0.05$ \\
\bottomrule
\end{tabular}
\end{table}

\begin{table*}[ht]
\centering
\caption{Performance comparison of the algorithms across five datasets, each with three ranks. Values represent the reconstruction loss.}
\label{tab:reconstruction loss}
{\small
\setlength{\tabcolsep}{1pt}
\begin{tabular}{l l c c c c}
\toprule
\textbf{Dataset} & \textbf{rank} & \shortstack{\textbf{Sub-}\\\textbf{gradient}} & \shortstack{\textbf{Variable}\\\textbf{smoothing}} & \shortstack{\textbf{MELMO}\\$(p = 7/12, q = 1/3)$} & \shortstack{\textbf{MELMO}\\$(p = 2/3, q = 1/4)$}\\
\midrule

\multirow{3}{*}{Camera} 
& 10 
& $1.449\times10^{-2}$
& $1.454\times10^{-2}$
& $1.420\times10^{-2}$
& $\textbf{1.419}\times10^{-2}$
\\
& 20 
& $8.346\times10^{-3}$
& $6.876\times10^{-3}$
& $6.648\times10^{-3}$
& $\textbf{6.634}\times10^{-3}$
\\
& 30 
& $7.158\times10^{-3}$
& $4.500\times10^{-3}$
& $3.748\times10^{-3}$
& $\textbf{3.727}\times10^{-3}$
\\
\midrule

\multirow{3}{*}{Spectrometer} 
& 10 
& $8.568\times10^{-3}$
& $2.310\times10^{-3}$
& $7.925\times10^{-4}$
& $\textbf{7.431}\times10^{-4}$
\\

& 20 
& $7.146\times10^{-3}$
& $3.240\times10^{-3}$
& $4.781\times10^{-4}$
& $\textbf{4.321}\times10^{-4}$
\\

& 30 
& $6.592\times10^{-3}$
& $3.152\times10^{-3}$
& $3.553\times10^{-4}$
& $\textbf{3.186}\times10^{-4}$
\\

\midrule

\multirow{3}{*}{Olivetti} 
& 10 
& $2.074\times10^{-2}$
& $2.118\times10^{-2}$
& $\textbf{2.049}\times10^{-2}$
& $2.375\times10^{-1}$
\\
& 20 
& $\textbf{1.434}\times10^{-2}$
& $1.850\times10^{-2}$
& $1.480\times10^{-2}$
& $2.230\times10^{-1}$
\\

& 30 
& $\textbf{1.115}\times10^{-2}$
& $1.716\times10^{-2}$
& $1.243\times10^{-2}$
& $1.717\times10^{-1}$
\\
\midrule

\multirow{3}{*}{Football} 
& 10 
& $4.540\times10^{-1}$
& $4.512\times10^{-1}$
& $4.499\times10^{-1}$
& $\textbf{4.497}\times10^{-1}$
\\

& 20 
& $3.208\times10^{-1}$
& $2.997\times10^{-1}$
& $2.983\times10^{-1}$
& $\textbf{2.980}\times10^{-1}$
\\

& 30 
& $2.336\times10^{-1}$
& $1.991\times10^{-1}$
& $1.981\times10^{-1}$
& $\textbf{1.977}\times10^{-1}$
\\
\midrule

\multirow{3}{*}{Low rank synthetic} 
& 10 
& $1.363\times10^{-2}$
& $9.610\times10^{-3}$
& $9.287\times10^{-3}$
& $\textbf{9.210}\times10^{-3}$
\\
& 20 
& $1.897\times10^{-2}$
& $9.380\times10^{-3}$
& $8.115\times10^{-3}$
& $\textbf{8.080}\times10^{-3}$
\\

& 30 
& $1.619\times10^{-2}$
& $5.782\times10^{-3}$
& $4.517\times10^{-3}$
& $\textbf{4.476}\times10^{-3}$
\\
\bottomrule
\end{tabular}
}
\end{table*}

\begin{figure}[ht!]
    \centering
    \begin{minipage}{0.49\textwidth}
        \centering
        \includegraphics[width=\linewidth]{Fk_football_0.3.png}\\
        (Football, $\omega=0.3$)
    \end{minipage}
    \hfill
    \begin{minipage}{0.49\textwidth}
        \centering
        \includegraphics[width=\linewidth]{Fk_miserables_0.3.png}\\
        (Misérables, $\omega=0.3$)
    \end{minipage}
    \hfill
    \begin{minipage}{0.49\textwidth}
        \centering
        \includegraphics[width=\linewidth]{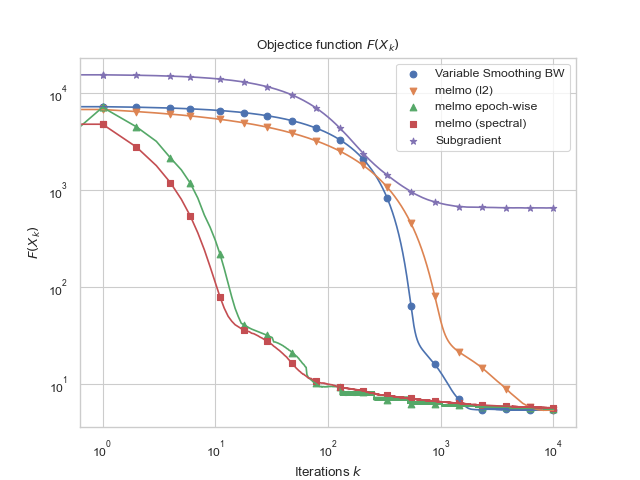}\\
        (Football, $\omega=0.5$)
    \end{minipage}
    \hfill
    \begin{minipage}{0.49\textwidth}
        \centering
        \includegraphics[width=\linewidth]{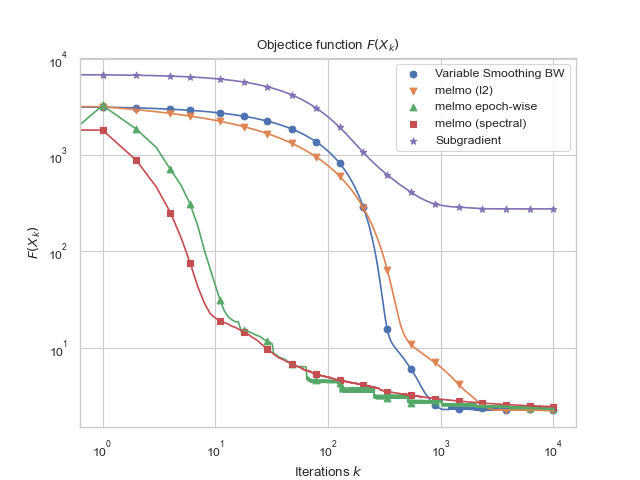}\\
        (Misérables, $\omega=0.5$)
    \end{minipage}
    \hfill
    \begin{minipage}{0.49\textwidth}
        \centering
        \includegraphics[width=\linewidth]{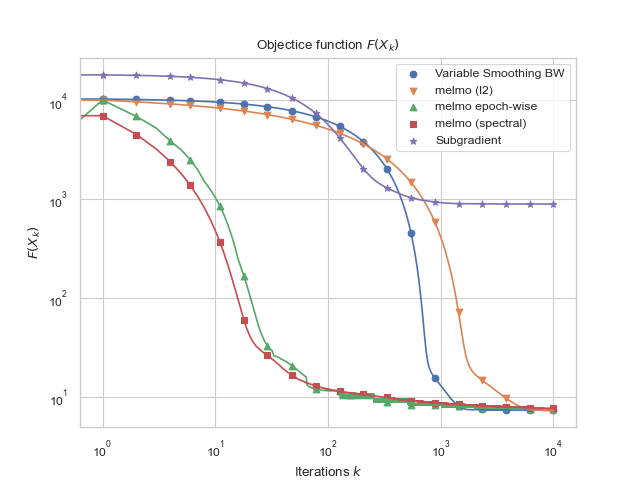}\\
        (Football, $\omega=0.7$)
    \end{minipage}
    \hfill
    \begin{minipage}{0.49\textwidth}
        \centering
        \includegraphics[width=\linewidth]{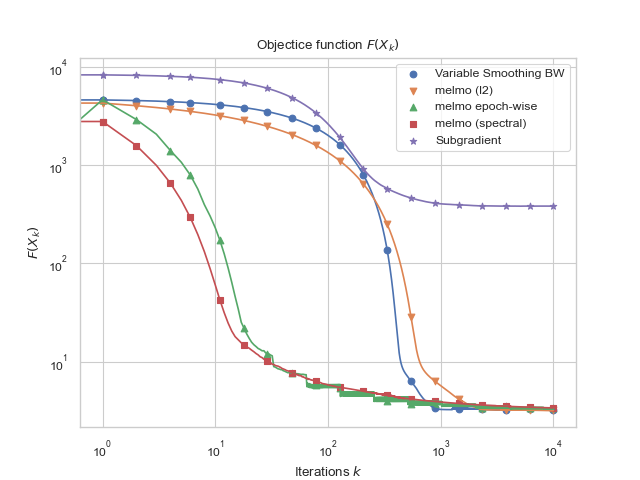}\\
        (Misérables, $\omega=0.7$)
    \end{minipage}
    \caption{Penalized objective $F$ vs.\ iteration (log scale) for masked matrix recovery on Football and Misérables datasets across three observation ratios.}
    \label{fig:masked_objectives}
\end{figure}

\begin{figure}[ht!]
    \centering
    \begin{minipage}{0.49\textwidth}
        \centering
        \includegraphics[width=\linewidth]{rLoss_football_0.3.png}\\
        (Football, $\omega=0.3$)
    \end{minipage}
    \hfill
    \begin{minipage}{0.49\textwidth}
        \centering
        \includegraphics[width=\linewidth]{rLoss_miserables_0.3.png}\\
        (Misérables, $\omega=0.3$)
    \end{minipage}
    \hfill
    \begin{minipage}{0.49\textwidth}
        \centering
        \includegraphics[width=\linewidth]{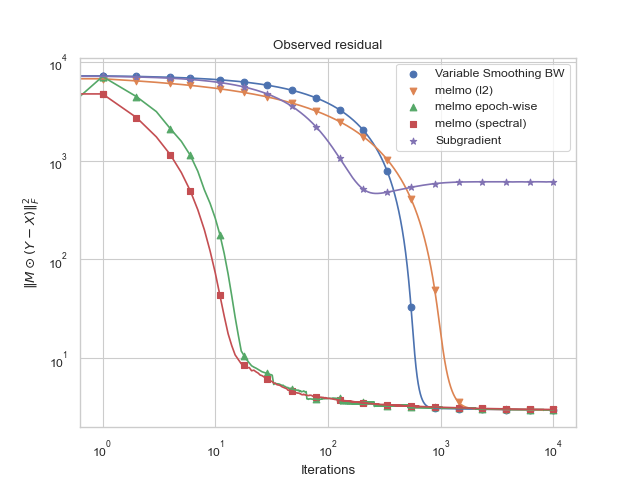}\\
        (Football, $\omega=0.5$)
    \end{minipage}
    \hfill
    \begin{minipage}{0.49\textwidth}
        \centering
        \includegraphics[width=\linewidth]{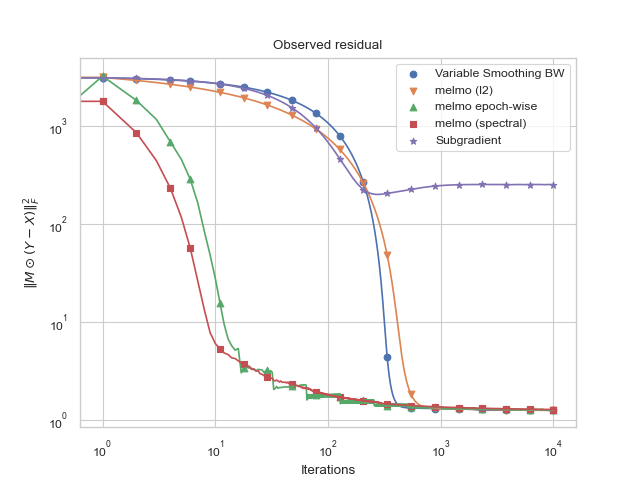}\\
        (Misérables, $\omega=0.5$)
    \end{minipage}
    \hfill
    \begin{minipage}{0.49\textwidth}
        \centering
        \includegraphics[width=\linewidth]{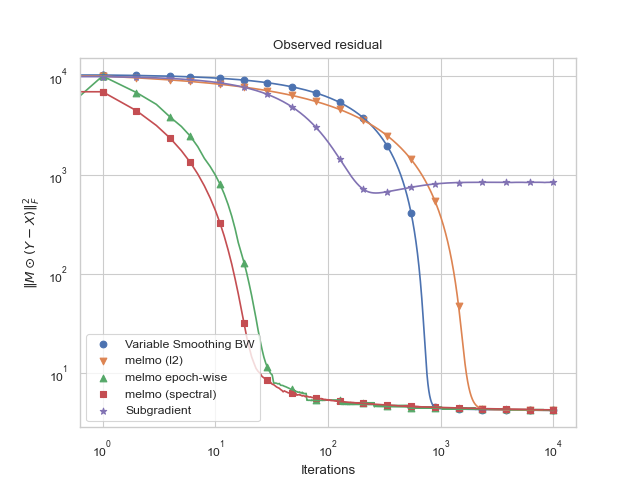}\\
        (Football, $\omega=0.7$)
    \end{minipage}
    \hfill
    \begin{minipage}{0.49\textwidth}
        \centering
        \includegraphics[width=\linewidth]{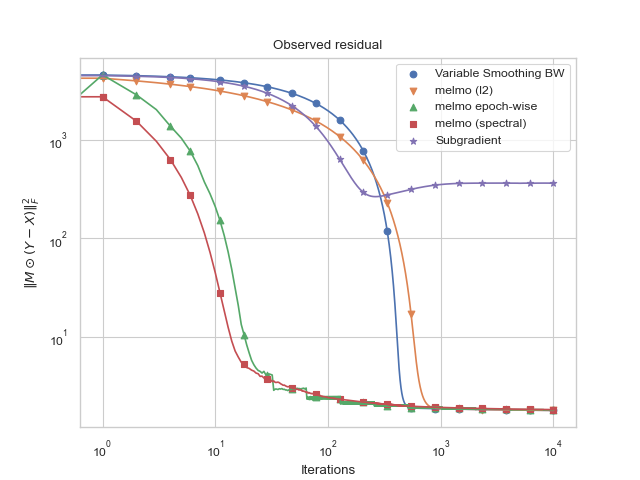}\\
        (Misérables, $\omega=0.7$)
    \end{minipage}
    \caption{Observed residual loss vs.\ iteration (log scale) for masked matrix recovery on Football and Misérables datasets across three observation ratios.}
    \label{fig:masked_rLoss}
\end{figure}

\begin{figure}[ht!]
    \centering
    \begin{minipage}{0.49\textwidth}
        \centering
        \includegraphics[width=\linewidth]{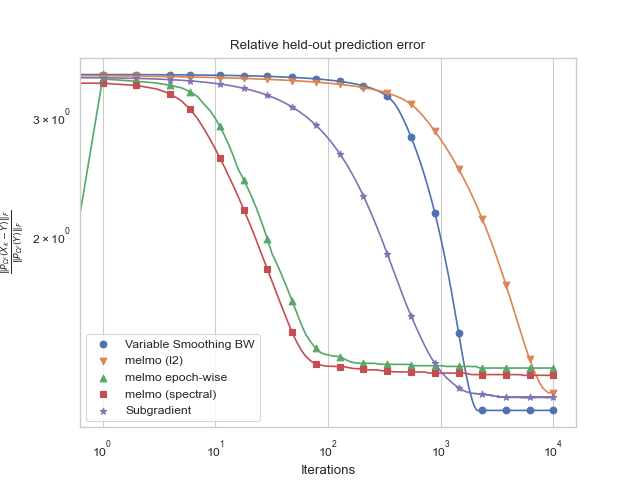}\\
        (Football, $\omega=0.3$)
    \end{minipage}
    \hfill
    \begin{minipage}{0.49\textwidth}
        \centering
        \includegraphics[width=\linewidth]{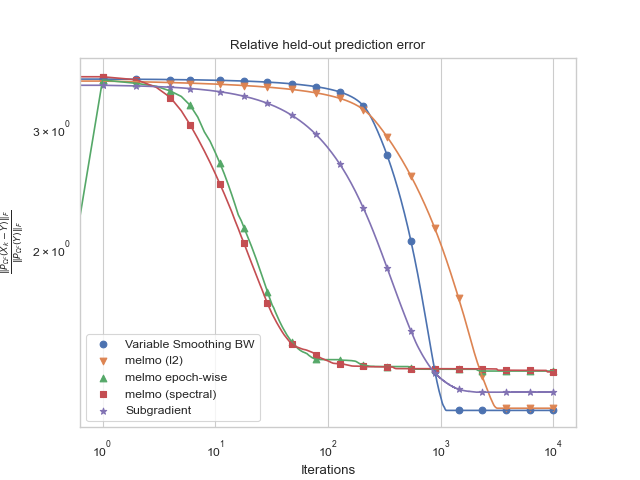}\\
        (Misérables, $\omega=0.3$)
    \end{minipage}
    \hfill
    \begin{minipage}{0.49\textwidth}
        \centering
        \includegraphics[width=\linewidth]{heldOutLoss_football_0.5.png}\\
        (Football, $\omega=0.5$)
    \end{minipage}
    \hfill
    \begin{minipage}{0.49\textwidth}
        \centering
        \includegraphics[width=\linewidth]{heldOutLoss_miserables_0.5.png}\\
        (Misérables, $\omega=0.5$)
    \end{minipage}
    \hfill
    \begin{minipage}{0.49\textwidth}
        \centering
        \includegraphics[width=\linewidth]{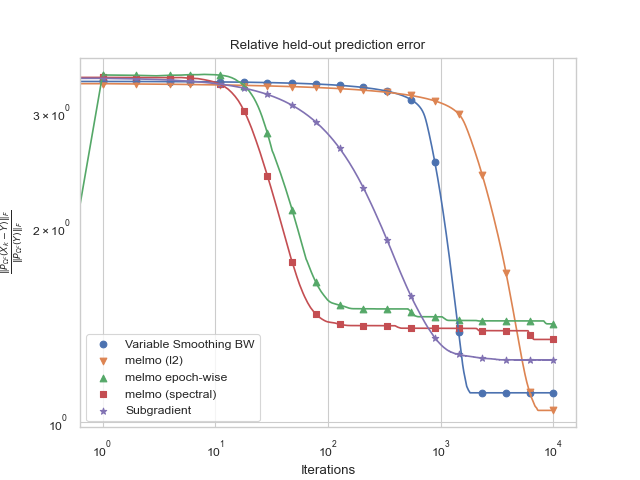}\\
        (Football, $\omega=0.7$)
    \end{minipage}
    \hfill
    \begin{minipage}{0.49\textwidth}
        \centering
        \includegraphics[width=\linewidth]{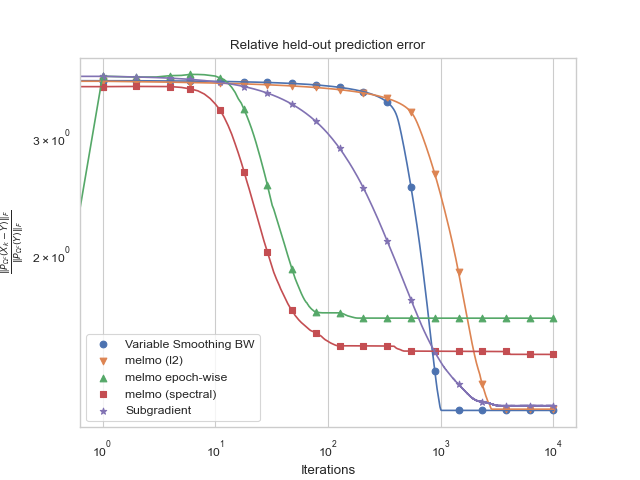}\\
        (Misérables, $\omega=0.7$)
    \end{minipage}
    \caption{Relative held-out loss vs.\ iteration (log scale) for masked matrix recovery on Football and Misérables datasets across three observation ratios.}
    \label{fig:masked_heldOutLoss}
\end{figure}

\begin{figure}[ht!]
    \centering
    \begin{minipage}{0.49\textwidth}
        \centering
        \includegraphics[width=\linewidth]{proximalGap_football_0.3.png}\\
        (Football, $\omega=0.3$)
    \end{minipage}
    \hfill
    \begin{minipage}{0.49\textwidth}
        \centering
        \includegraphics[width=\linewidth]{proximalGap_miserables_0.3.png}\\
        (Misérables, $\omega=0.3$)
    \end{minipage}
    \hfill
    \begin{minipage}{0.49\textwidth}
        \centering
        \includegraphics[width=\linewidth]{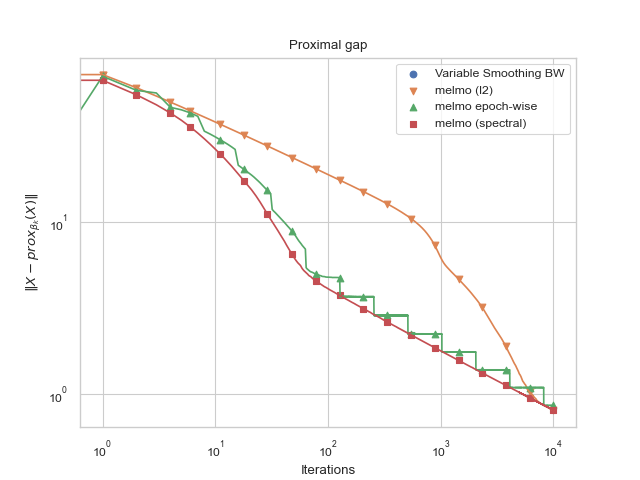}\\
        (Football, $\omega=0.5$)
    \end{minipage}
    \hfill
    \begin{minipage}{0.49\textwidth}
        \centering
        \includegraphics[width=\linewidth]{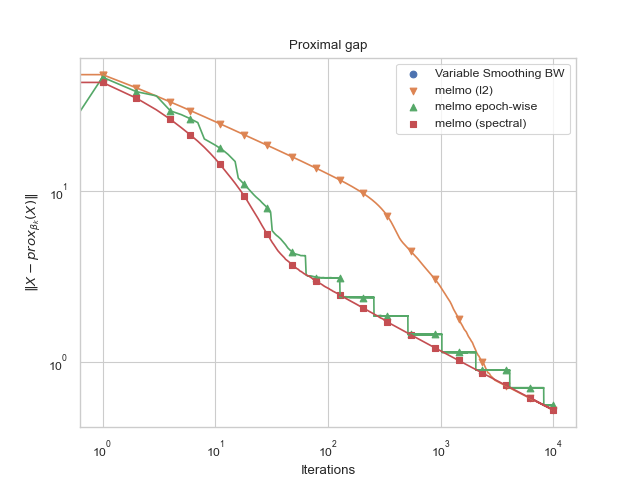}\\
        (Misérables, $\omega=0.5$)
    \end{minipage}
    \hfill
    \begin{minipage}{0.49\textwidth}
        \centering
        \includegraphics[width=\linewidth]{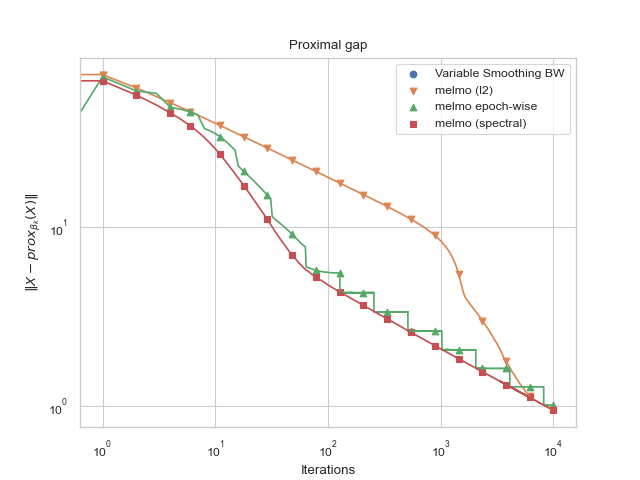}\\
        (Football, $\omega=0.7$)
    \end{minipage}
    \hfill
    \begin{minipage}{0.49\textwidth}
        \centering
        \includegraphics[width=\linewidth]{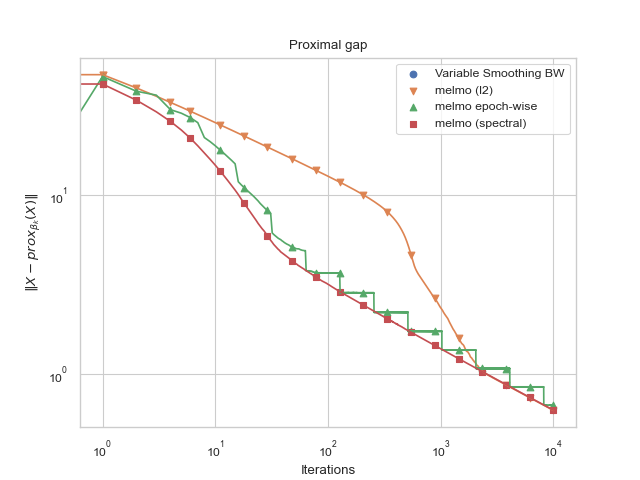}\\
        (Misérables, $\omega=0.7$)
    \end{minipage}
    \caption{Proximal gap vs.\ iteration (log scale) for masked matrix recovery on Football and Misérables datasets across three observation ratios.}
    \label{fig:masked_proximalGap}
\end{figure}

\begin{figure}[ht!]
    \centering
    \begin{minipage}{0.49\textwidth}
        \centering
        \includegraphics[width=\linewidth]{dualNorms_football_0.3.png}\\
        (Football, $\omega=0.3$)
    \end{minipage}
    \hfill
    \begin{minipage}{0.49\textwidth}
        \centering
        \includegraphics[width=\linewidth]{dualNorms_miserables_0.3.png}\\
        (Misérables, $\omega=0.3$)
    \end{minipage}
    \hfill
    \begin{minipage}{0.49\textwidth}
        \centering
        \includegraphics[width=\linewidth]{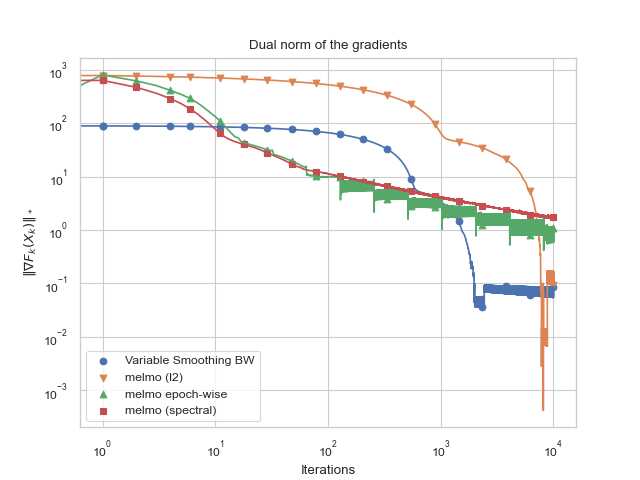}\\
        (Football, $\omega=0.5$)
    \end{minipage}
    \hfill
    \begin{minipage}{0.49\textwidth}
        \centering
        \includegraphics[width=\linewidth]{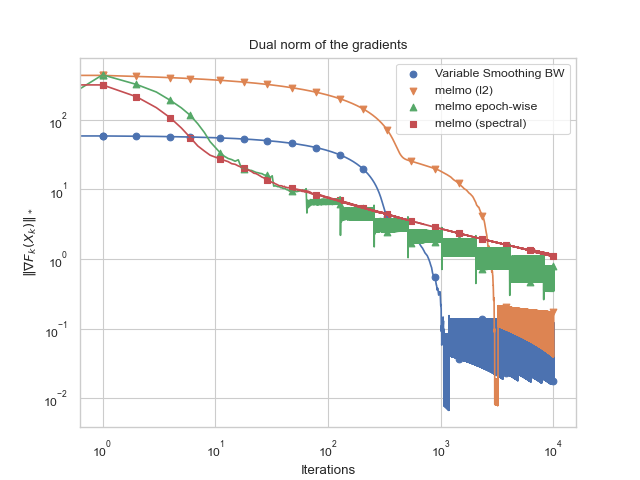}\\
        (Misérables, $\omega=0.5$)
    \end{minipage}
    \hfill
    \begin{minipage}{0.49\textwidth}
        \centering
        \includegraphics[width=\linewidth]{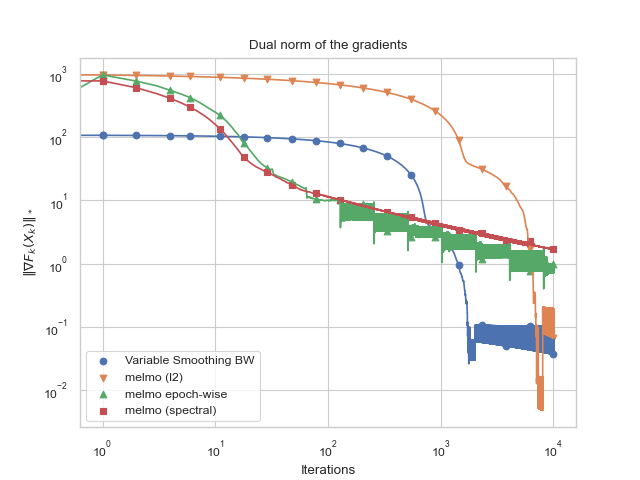}\\
        (Football, $\omega=0.7$)
    \end{minipage}
    \hfill
    \begin{minipage}{0.49\textwidth}
        \centering
        \includegraphics[width=\linewidth]{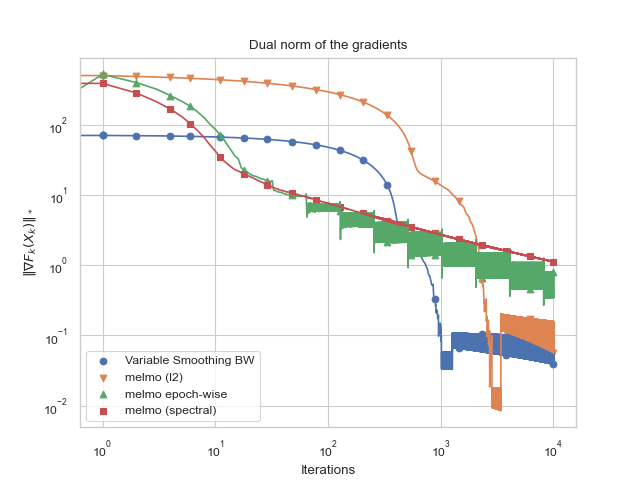}\\
        (Misérables, $\omega=0.7$)
    \end{minipage}
    \caption{Dual norm of the gradients vs.\ iteration (log scale) for masked matrix recovery on Football and Misérables datasets across three observation ratios.}
    \label{fig:masked_dualNorms}
\end{figure}
\end{document}